\documentclass[onethmnumfalse]{siamltex}

\usepackage{amsmath,amsfonts,amssymb}
\usepackage{latexsym}
\usepackage{epsfig,overpic}
\usepackage{commath}
\usepackage{mathtools}

\usepackage{subcaption}
\usepackage[utf8]{inputenc}
\usepackage{booktabs}
\usepackage{csvsimple}
\usepackage{colortbl}

\let\pa\partial

\newcommand{\R}{\mathbb R}

\newcommand{\ga}{\mathbf A}

\newcommand{\bD}{\mathbf D}
\newcommand{\gd}{\mathbf D}

\newcommand{\bF}{\mathbf F}

\newcommand{\gh}{\mathbf H}
\newcommand{\bI}{\mathbf I}

\newcommand{\bP}{\mathbf P}
\newcommand{\gp}{\mathbf P}

\newcommand{\bV}{\mathbf V}
\newcommand{\bU}{\mathbf U}
\newcommand{\bSigma}{\mathbf \Sigma}

\newcommand{\bn}{\mathbf n}
\newcommand{\n}{\mathbf n}
\newcommand{\bp}{\mathbf p}
\newcommand{\f}{\mathbf f}
\newcommand{\p}{\mathbf p}

\newcommand{\bs}{\mathbf s}
\newcommand{\s}{\mathbf s}
\newcommand{\bu}{\mathbf u}
\newcommand{\bv}{\mathbf v}

\renewcommand{\v}{\ensuremath{\mathbf{v}}}
\newcommand{\bw}{\mathbf w}
\newcommand{\w}{\mathbf w}
\newcommand{\bx}{\mathbf x}
\newcommand{\x}{\mathbf x}
\newcommand{\by}{\mathbf y}
\newcommand{\bz}{\mathbf z}
\newcommand{\zz}{\ensuremath{\mathbf{z}}}

\newcommand{\T}{\mathcal T}

\newcommand{\Tau}{\mathcal T}

\newcommand{\spa}{s}

\newcommand{\ti}{q}

\newcommand{\stab}{\operatorname{stab}}

\newcommand{\bnu}{\boldsymbol{\nu}}
\newcommand{\Div}{\operatorname{div}}
\renewcommand{\div}{\operatorname{div}}

\newcommand{\tr}{\mathop{\rm tr}}
\newcommand{\id}{\mathop{\rm id}\nolimits}

\newcommand\restrict[1]{\raisebox{-.5ex}{$|$}_{#1}}
\newcommand{\lin}{\operatorname{lin}}

\newcommand{\triplenorm}{\ensuremath{| \! | \! |}}
\newcommand*{\ddt}[3][]{\ensuremath{\frac{\partial^{#1} #2}{\partial #3}}}

\usepackage{tikz}
\usetikzlibrary{intersections}
\usetikzlibrary{arrows}
\usetikzlibrary{patterns}
\usetikzlibrary{calc} 
\usetikzlibrary{quotes,babel,angles}
\usetikzlibrary{decorations.pathreplacing}
\usetikzlibrary{pgfplots.groupplots}
\usepackage{xintexpr}
\usepackage{pgfplots}

\usepackage{hyperref}
\usepackage{cleveref}
\newtheorem{remark}[theorem]{Remark}
\newtheorem{assumption}[theorem]{Assumption}

\crefname{equation}{}{}

\pgfplotscreateplotcyclelist{haukes_cycle_list}{
	blue,every mark/.append style={fill=blue},mark=square*\\
	red,every mark/.append style={fill=red},mark=triangle*\\
	green,every mark/.append style={fill=green},mark=diamond*\\
	orange,every mark/.append style={fill=orange},mark=pentagon*\\
	purple,every mark/.append style={fill=purple},mark=x\\
	brown,every mark/.append style={fill=brown},mark=*\\
}
\pgfplotsset{cycle list name=haukes_cycle_list}
\pgfplotsset{compat=1.15}
\pgfplotsset{	every non boxed x axis/.style={} }

\begin{document}

\title{Analysis of a space-time unfitted finite element method for {PDEs} on evolving surfaces}
	\author{Arnold Reusken\thanks{Institut f\"ur
		Geometrie und Praktische  Mathematik, RWTH-Aachen University, D-52056 Aachen,
		Germany; email: {\tt reusken@igpm.rwth-aachen.de}}
	  \and Hauke Sass\thanks{Institut f\"ur
		Geometrie und Praktische  Mathematik, RWTH-Aachen University, D-52056 Aachen,
		Germany; email: {\tt sass@igpm.rwth-aachen.de}}  }

\maketitle

\begin{abstract}
  In this paper we  analyze a space-time unfitted finite element method for the discretization of scalar surface partial differential equations on evolving surfaces. For higher order approximations of the evolving surface we use the technique of (iso)parametric mappings for which a level set representation of the evolving surface is essential. We derive basic results in which certain  geometric characteristics of the exact space-time surface are related to corresponding ones of the numerical surface approximation. \textcolor{black}{These results are used in an error analysis of a higher order space-time TraceFEM.}  
  \end{abstract}



\section{Introduction}
In the past two decades a large toolbox of finite element methods for solving \emph{scalar} partial differential equations (PDEs) on \emph{evolving} surfaces has been developed.  The probably most  prominent method is the evolving surface finite element method (ESFEM) \cite{DEreview,DziukElliot2013a,kov2018}. This is an elegant and  popular method for which a complete error analysis has been developed \cite{elliott2015error,DziukElliot2013a}. A key feature of this method is that it is based on a Lagrangian approach in which a triangulation of the initial surface and a corresponding finite element space are transported along the flow lines of a (given) velocity field. This makes the method  attractive for problems with smoothly and slowly varying surfaces. There are also recent contributions  \cite{Zhao_2020,kovacs2022numerical} in which the ESFEM  is applied to problems with geometric singularities (e.g. droplet pinch-off). It is, however, well-known that such Lagrangian techniques have drawbacks if the geometry of the surface is strongly varying on the relevant time scales or if there are topological singularities, e.g., merging or splitting phenomena. Another finite element technique for PDEs on evolving surfaces is introduced in \cite{LOX_SIAM_2017}. This method is based on an Eulerian approach in which  the stabilized trace finite element method (TraceFEM) is used for  spatial discretization and combined with standard finite differences (BDF) for the time discretization. With such an Eulerian approach it is easier to handle topological singularities. Furthermore,  if the surface PDE is coupled with a bulk PDE, the use of TraceFEM, which is based on standard bulk finite element spaces, may lead to advantages concerning the implementation. In \cite{LOX_SIAM_2017} an error analysis of this method is presented for BDF1 combined with higher order finite elements for spatial discretization. In this analysis geometry errors are treated but  there is an issue concerning sufficiently accurate numerical integration on the approximate surface~\cite[Remark 4.1]{LOX_SIAM_2017}.  Yet another approach, also of Eulerian type, has been introduced in \cite{olshanskii2014eulerian,olshanskii2014error} and uses the TraceFEM or CutFEM principle not only in space but in space-time. In other words, the discretization of the evolving surface PDE uses a restriction to the space-time surface of standard space-time finite element spaces on a fixed bulk (tensor product) space-time mesh. In the framework of CutFEM this is a natural approach for this class of partial differential equations.  The structure of this space-time TraceFEM is such that it naturally fits to a level set representation of the evolving surface. Such level set representations are often used for problems with strongly varying surface geometries or with topological singularities.  

In this paper we treat this space-time TraceFEM. Its main characteristics are the following. The method is of Eulerian type, using a (per time step) fixed unfitted bulk spatial triangulation and corresponding standard  space-time (tensor product)  finite element spaces. As input for the method one needs  a level set representation of the evolving surface. If a sufficiently accurate (made precise further on in the paper) level set function approximation is available  and higher order bulk space-time finite element spaces are used, the resulting surface PDE discretization method is also of (optimal) higher order accuracy for smoothly varying surfaces with sufficiently smooth solutions. Without any modifications the method can also be used for the discretization of problems with topological singularities. Of course, in such singular cases one can not expect global higher order convergence. In \cite{olshanskii2014eulerian,olshanskii2014error} the method was introduced for piecewise linears (in space and time) and the issue of geometry approximation was not considered. In the recent paper \cite{sass2022} the method was extended in two directions. Firstly, a 
	higher order variant was introduced based on the parametric TraceFEM \cite{lehrenfeld2016high}. Secondly,  the so-called volume normal derivative stabilization was included. This stabilization is important, in particular for higher order finite elements, to control the conditioning of the resulting stiffness matrix. In \cite{sass2022} these extensions are explained and a systematic numerical study of this method is presented that demonstrates (optimal) higher order accuracy for smoothly varying surfaces with sufficiently smooth solutions and its robustness with respect to   topological singularities.
	
	 An important   topic of the present paper is an error analysis of this higher order space-time TraceFEM.  
	 We put the analysis in a rather general framework in the following sense. We assume that the evolving surface is characterized using a level set function $\phi=\phi(\bx,t)$,  $\bx \in \Omega \subset  \R^3$, $t \in [0,T]$. The smooth space-time surface, denoted by $S \subset \R^4$, is the zero level of $\phi(\cdot,t)$ for $t \in [0,T]$.  For an efficient construction of  a (higher order) approximation $S_h$ of $S$ we use the general technique of level set based parametric mappings.  In \cite{lehrenfeld2016high} this technique is introduced for the case of a stationary surface and in \cite{HLP2022} its generalization to space-time surfaces is treated. 
	 
	 In the first part of this paper we derive basic properties of this space-time surface approximation. More specifically, we examine relations between surface measures on $S_h$ and $S$, derive useful integral transformation and partial integration formulas on $S_h$ mirroring those on $S$, and show how to transform the space-time surface gradient of a function on $S$ to the space-time surface gradient on  $S_h$ of an extension of this function. We also derive estimates on the change in surface measure when moving from $S$ to $S_h$ and on the size of conormal jumps in the (not necessarily shape regular) surface triangulation of $S_h$.  For these geometric estimates we use results that are presented in the recent work \cite{HeimannLehrenfeld}
	 
	  The results of the first part of the paper, which depend essentially only on the level set function approximation $\phi_h \approx \phi$ and on properties of the space-time parametric mapping \cite{HLP2022,HeimannLehrenfeld}, may be useful also in the development and analysis of methods other than space-time TraceFEM. In the second part of the paper we give an error analysis of the space-time TraceFEM introduced in \cite{sass2022}. In this analysis we use many of the basic results derived in the  first part. We do not include results of numerical experiments with space-time TraceFEM. For this we refer to \cite{sass2022,Diss_Sass}.

The paper is organized as follows. In Sections~\ref{secLS} and \ref{secSurface} we collect basic facts of the level set function approximation $\phi_h \approx \phi$ and recall relevant properties of the parametric mapping that is used for the mesh deformation. We also explain how functions defined on $S$ or $S_h$ are extended to a small neighborhood. In  Section~\ref{section_basic_geometry_est} we derive several useful integral transformation rules  and a partial integration rule on $S_h^n$, which denotes the $n$th time slab part of $S_h$. We also show how for a function defined on $S$ its space-time surface gradient can be expressed in terms of a space-time  surface gradient on $S_h$ of an extension of this function. 
In Section~\ref{secGeometry} bounds for certain geometric quantities, e.g., the change in surface measure between $S$ and $S_h$ or the jump in certain conormals on $S_h$ are derived. In Section~\ref{section_application} we consider the standard model of a surfactant equation on a smoothly evolving surface and outline the space-time TraceFEM applied to this surface PDE. In Section~\ref{consistency} we present an error analysis of this method. This analysis has a standard structure, based on a space-time stability estimate, a Strang Lemma, consistency estimates and interpolation error bounds. The main result is given in Theorem~\ref{mainthm}, which contains optimal order error bounds for the space-time TraceFEM. 

\section{Level set representation of the continuous space-time surface} \label{secLS}
We introduce a space-time surface that is characterized using a level set function. For a domain 
$\Omega\subset \R^3$, $T>0$, a smooth velocity field $\mathbf{w}:\Omega\times [0,T]\rightarrow \R^3$ and an initialization $\phi_0: \Omega \to \R$, assume that  the level set function $\phi: \Omega\times [0,T] \to \R$ solves the level set equation 
$\frac{\partial \phi}{\partial t} + \bw \cdot \nabla \phi =0$ and satisfies $\phi(\cdot,0)=\phi_0$. Furthermore we assume that the zero level 
\begin{equation*}
	\Gamma(t)=\left\lbrace \bx\in\Omega: \phi(\bx,t)=0 \right\rbrace
\end{equation*}
is  a closed, connected and orientable $C^2$-hypersurface for all $t \in [0,T]$. This in particular implies that the zero level $\Gamma(0)$ is  advected by $\mathbf{w}$.  
The corresponding smooth space-time surface   is denoted by
\begin{equation*}
	S\coloneqq \bigcup_{t\in [0,T]}\left(\Gamma(t)\times \{t\}\right)\subset\R^4.
\end{equation*}
\textcolor{black}{In the following, we use \(\lesssim\) and \(\gtrsim\) for inequalities that hold with constants independent of mesh size, time step size and cut configuration within the mesh. We use \(\sim\) to denote that both \(\lesssim\) and \(\gtrsim\) hold.
} 
\begin{assumption}\label[assumption]{a1}
We assume  the following standard properties of a level set function. For all $(\bx,t)$ in a neighborhood $U \subset \R^4$ of $S$:
\begin{equation}
	\norm{\nabla\phi(\bx,t)}\sim 1, \quad \norm{D^2\phi(\bx,t)}\lesssim 1,\quad \norm{\partial_t \phi}\lesssim 1, \quad \norm{\partial_t \nabla \phi}\lesssim 1, \label{phi_assumptions}
\end{equation}
 and for all $t \in [0,T]$, $\epsilon, \tilde{\epsilon} \in [0,\epsilon_0]$, $\epsilon_0 > 0$ sufficiently small:
\begin{equation} \label{AA1}
	\abs{\phi(\x + \epsilon \nabla\phi(\x),t)- \phi(\x+\tilde{\epsilon}\nabla\phi(\x),t)}\sim \abs{\epsilon - \tilde{\epsilon}}, \quad \bx \in U(t),
\end{equation}
where  $U(t) \subset \Omega$ is such that $U= \cup_{t \in [0,T]} U(t) \times \{t\}$.
\end{assumption}

The spatial unit outer normal $\bn:S \to \R^3$  and the  space-time unit outer normal $\bn_S: S \to \R^4$ 
are given by
\begin{equation} \label{A1}
	\mathbf{n}\coloneqq \frac{\nabla \phi}{\norm{\nabla \phi}} \quad \text{and}\quad \mathbf{n}_S\coloneqq \frac{\nabla_{(\x,t)}\phi}{\norm{\nabla_{(\x,t)}\phi}}.
\end{equation} 
On $S$ the following relation between $\bn$ and $\bn_S$ holds:
\begin{equation}
	\n_S=\frac{1}{\alpha}\begin{pmatrix}\mathbf{n}\\-V\end{pmatrix},\quad V\coloneqq \mathbf{w}\cdot \mathbf{n},\quad \alpha\coloneqq \sqrt{1+V^2}. \label{ns_schreibweise}
\end{equation}
Note that  $V=\bw\cdot\bn$ is the normal velocity of the evolving surface $\Gamma(t)$ and satisfies
\begin{equation} \label{formnormalvel}
V= \frac{-1}{\norm{\nabla\phi}}\frac{\partial \phi}{\partial t}.
\end{equation}
 Let $\dif s$ and $\dif \sigma$ be the surface measures on $\Gamma(t)$ and $S$ respectively. The integral transformation formula
\begin{equation}
	\int_0^T\int_{\Gamma(t)}g\dif s\dif t=\int_S \frac{g}{\alpha}\dif \sigma, \quad g\in L^2(S),\label{transformationsformelMAIN}
\end{equation}
holds,  cf. \cite[Appendix A.1]{Diss_Sass} for a proof. {\color{black} Clearly we could also use the transformation formula \eqref{transformationsformelMAIN} in a different form, with $\alpha \, g$ and $g$ on the left-hand side and right-hand side, respectively. We use the form \eqref{transformationsformelMAIN} because in the application in Section~\ref{section_application} we use tensor product spaces, which fit naturally to the space-time integral (without any scaling) as in the left-hand side of \eqref{transformationsformelMAIN}.} One important result in \Cref{section_basic_geometry_est} (Corollary~\ref{fancyformelcoro}) is an  analogue of \cref{transformationsformelMAIN} for a space-time surface approximation that has only Lipschitz smoothness. 
\section{Approximation of the space-time surface} \label{secSurface}
In the context of discretization of surface partial differential equations on evolving surfaces it is natural to use a time stepping procedure. 
For $N\in \mathbb{N}$, let the time interval $[0,T]$ be partitioned into smaller time intervals $I_n\coloneqq [t_{n-1},t_n]$, $1 \leq n \leq N$, where $0=t_0<t_1<\dots<t_N=T$.  We define $Q\coloneqq \Omega\times [0,T]\subset \mathbb{R}^{4}$ and for $n=1, \ldots, N$, we define the space-time cylinders and corresponding space-time surface
\begin{equation*}
	 Q_n\coloneqq \Omega \times I_n, \quad S^n\coloneqq \bigcup_{t\in I_n}( \Gamma(t)\times \{t\}).
\end{equation*}
Let $\T$ be an element of a family $\{{\T}_h\}_{h>0}$ of shape regular tetrahedral triangulations of $\Omega$,  the triangulation $Q_{h,n}\coloneqq \T\times I_n$ divides the time slab $Q_n$ into space-time prismatic elements. 
The  space-time triangulation of $Q$ is denoted by ${Q_h\coloneqq \bigcup_{n=1}^N Q_{h,n}}$. To simplify the presentation and the analysis we do not allow a change of triangulation between time slabs.  
\textcolor{black}{Let $V_h^{m}$, $m\in \mathbb{N}$, be the standard $H^1(\Omega)$-conforming finite element space of piecewise polynomials up to degree $m$ on the  triangulation $\T$}. For $m_{\spa},m_{\ti}\in \mathbb{N}$ a space-time finite element product space is given  by 
\begin{equation}
	V_h^{m_{\spa}, m_{\ti}}\coloneqq \big\{ v:Q \rightarrow \R : v_{|Q_n}(\bx,t)=\sum_{i=0}^{m_{\ti}}t^i v_i(\bx), ~~v_i \in V_h^{m_{\spa}},~ 1 \leq n \leq N\big\}.\label{st_fe_space}
\end{equation}
Note that functions from $V_h^{m_{\spa},m_{\ti}}$ are continuous in space but may be discontinuous in time  at the time interval boundaries $t_n$. 
This finite element space  is used both in the finite element discretization of the surface  PDE studied in Section~\ref{section_application} and, with possibly different degrees $m_{\spa}, m_{\ti}$,  in the parametric mapping used for the higher order geometry approximation. For the geometry approximation we assume a $\phi_h \in V_h^{k_{g,\spa}, k_{g,\ti}}$, i.e., with degree $k_{g,\spa}$ in $\bx$ and $k_{g,\ti}$ in $t$. 
Such a  $\phi_h$ may be an
interpolation of $\phi$, if the latter is available. 
In applications, $\phi$ may be determined implicitly by a level set equation.   
For a higher order surface approximation we need a sufficiently accurate approximation $\phi_h \approx \phi$. Requirements for the accuracy of $\phi_h$ as approximation of $\phi$ are specified below, cf. Assumption~\ref{ass_phi_h}.

The evaluation of integrals on the zero level of a higher order polynomial (e.g., $\phi_h$) requires (too) much computational effort.  This motivates the use of an approach that is similar to the classical parametric finite element method. Below we explain the basic structure of the space-time parametric technique introduced and analyzed in \cite{HLP2022,HeimannLehrenfeld}. 
For further details we refer to this literature and to \cite{preuss2018higher,Diss_Sass,sass2022}.
\textcolor{black}{\begin{remark} \label{FEblend} \rm
                  In \cite{HLP2022,HeimannLehrenfeld} two different techniques for construction of the space-time parametric mapping  are presented, one based on the  \textit{FE blending} and the other one based on a \textit{smooth blending}. For the purpose of space-time discretization of surface PDEs the technique based on FE blending is the most natural (and simpler) choice, cf. \cite[Remark 3.1]{HeimannLehrenfeld}. In the present paper we therefore restrict to the construction of $\Theta_h^n$ based on FE blending.
                 \end{remark}
}
 \ \\[1ex]
\textcolor{black}{By $\mathcal{I}_{m}$ we denote the spatial nodal interpolation operator in $V_h^m$}.
 The spatially piecewise linear nodal interpolation is denoted by $\hat{\phi}_h\in V_h^{1,k_{g,\ti}}$: 
 \begin{equation}
 	\hat{\phi}_h(\cdot,t)\coloneqq \mathcal{I}_1\phi_h(\cdot,t)\in V_h^1, \quad  t\in [0,T],\label{phi_h_dach_definition}
 \end{equation}
 and  its corresponding piecewise planar zero level  at time $t$
 \begin{equation*}
 	\Gamma_{\lin}(t)\coloneqq \left\lbrace\bx\in \Omega : \hat{\phi}_h(\bx,t)=0\right\rbrace,
 \end{equation*}
 which is easy to compute.
 This surface  $\Gamma_{\lin}(t)$ is a (only)  second order accurate approximation of $\Gamma(t)$. The corresponding discrete space-time surface is denoted by
 \begin{equation*}
 S^n_{\lin} \coloneqq \left\lbrace (\bx,t)\in Q_n: \hat{\phi}_h(\bx,t)=0\right\rbrace,  \quad S_{\lin}\coloneqq \bigcup_{n=1}^N S_{\lin}^n.
 \end{equation*}
 We consider all elements that are cut by the piecewise linear surface at a point in time within one time slab and the corresponding space-time prisms, i.e., 
 \begin{align*}
 	\T_n^\Gamma & \coloneqq \left\lbrace K\in \mathcal{T}:  \mathrm{meas}_2((K\times I_n)\cap \Gamma_{\lin}(t))>0 \text{ for a }t\in I_n\right\rbrace, \\ 
 	Q_{h,n}^S & \coloneqq\T_n^{\Gamma}\times I_n, 
 	\quad n\in \{1,\dots,N\},
 \end{align*}
 and $Q_{h}^S:= \cup_{n=1}^N Q_{h,n}^S$. The corresponding domains are denoted by $\Omega_n^\Gamma \coloneqq \{\bx\in K: K\in \T_n^{\Gamma}\}$, $Q^S_n=\Omega_n^\Gamma \times I_n $ and $Q^S =\cup_{n=1}^N Q^S_n$. In the remainder we assume $Q^{S}\subset U$ (which is satisfied for $h$ and $\Delta t$ sufficiently small).
 
 The construction of the mesh deformation is based on an ``ideal'' spatial mapping $\Psi_t: U(t) \to \R^3$ that depends on the level set function $\phi$ and  for which $\Psi_t\big(\Gamma_{\lin}(t)\big)= \Gamma(t)$ holds. A corresponding space-time mapping is given by $\Psi: \, U \to \R^4$, $\Psi(\bx,t):= (\Psi_t(\bx),t)$. This mapping $\Psi$ is continuous on $U$. \textcolor{black}{By $\mathcal{P}^{k_{g,\ti}}(I_n)$ we denote the space of polynomials up to degree $k_{g,\ti}$ on $I_n$.}
 Let $\tau_m^n \in I_n$, $m=0,\dots, k_{g,\ti}$, be discrete points and $\mathcal{X}_{\tau_m^n} \in \mathcal{P}^{k_{g,\ti}}(I_n)$  corresponding finite element nodal basis functions.  On $I_n$ the function $\Psi_t$ is interpolated in $t$ by
 $\tilde \Psi_t(\bx)= \sum_{m=0}^{k_{g,\ti}} \mathcal{X}_{\tau_m^n}(t) \Psi_{\tau_m^n}(\bx)$. 
 For given $t = \tau_m^n$ a mapping  $\Theta_{h,t}^n \approx \Psi_t$ is constructed. This mapping, which deforms all elements of the triangulation $\T_n^{\Gamma}$ and \emph{depends (only) on $\phi_h$}, is such that $\Theta_{h,t}^n\in  \big(V_h^{k_{g,\spa}})^3$.
 The spatial mesh deformation  ${\Theta_{h,s}^{n}\in (V_h^{k_{g,\spa},k_{g,\ti}})^3}$ and the corresponding space-time mapping $\Theta_h^n$ are defined by
 \begin{equation}
 \Theta_{h,\spa}^{n}(\bx,\, t):=	\sum_{m=0}^{k_{g,\ti}}\mathcal{X}_{\tau^n_m}(t)\Theta_{h,\tau^n_m}^{n}(\bx), \quad     \Theta_{h}^{n}(\bx,t) \coloneqq	\big(\Theta_{h,\spa}^{n}(\bx,\, t), t\big), \quad (\bx,t)\in Q_{n}^S,\label{Theta_Def}
 \end{equation}
 cf. Fig.~\ref{theta_picture} for an illustration.
 \input{theta_picture}
 
 We note that opposite to the ideal mapping $\Psi$, the mesh deformation $\Theta_{h}^{n}$ is \emph{not} necessarily continuous between time slabs, cf. Remark~\ref{remconsist}.
Using this mesh deformation we obtain the \emph{space-time surface approximation}
 \begin{align} \label{Sapprox}
 	S_h^n&\coloneqq \left\lbrace (\bx,t)\in Q_n: (\bx,t) \in \Theta_h^n(S^n_{\lin})\right\rbrace, ~1 \leq n \leq N,\quad S_{h}\coloneqq \bigcup_{n=1}^N S_h^n. 
 \end{align}
 The accuracy of the approximation $S_h \approx S$ is  discussed in Section~\ref{secGeometry}. 
 The corresponding time slices of $S_h^n$ are denoted by
 \begin{equation}
 	\Gamma_h^n(t)\coloneqq \left\lbrace\bx\in\R^3: (\bx,t)\in S_h^n\right\rbrace,\quad t\in I_n, \quad n=1,\dots,N.\label{gamma_h_def}
 \end{equation}
For a given $n$ and almost everywhere on $Q_{n}^S$ we define (with $D\Theta$ the spatial Jacobian of $\Theta$)
\begin{equation}
	\bn_{\lin} \coloneqq  \frac{\nabla\hat{\phi}_h}{\norm{\nabla\hat{\phi}_h}}, \quad \bn_h\circ \Theta_h^n\coloneqq \frac{(D\Theta_{h,\spa}^n)^{-T}\bn_{\lin}}{\norm{(D\Theta_{h,\spa}^n)^{-T}\bn_{\lin}}},  
	\label{nh_defs}
\end{equation}
with $\hat{\phi}_h$ as defined in \cref{phi_h_dach_definition} and $\Theta_{h,\spa}^n$ as in \Cref{Theta_Def}.
Restricted to $\Gamma_{\rm lin}(t)$ these vectors are indeed the unit (spatial) normal vectors to the surface approximations $\Gamma_{\rm lin}(t)$ and $\Gamma_h^n(t)$, respectively.	On $Q_{n}^S$ we also introduce space-time variants:
\begin{equation}
	\bn_{S_{\lin}} \coloneqq  \frac{\nabla_{(\bx,t)} \hat{\phi}_h}{\norm{\nabla_{(\bx,t)} \hat{\phi}_h}},\quad 
	\bn_{S_{h}}\circ\Theta_h^n \coloneqq  \frac{(D_{(\bx,t)}\Theta_h^n)^{-T}\bn_{S_{\lin}}}{\norm{(D_{(\bx,t)}\Theta_h^n)^{-T}\bn_{S_{\lin}}}} .\label{nslin_def_anddefintion_nsh_higherorder}
\end{equation}
On $S_{\rm lin}^n$  these normals are the unit space-time normals to the respective space-time surfaces, cf.  \cite[Subsection 4.2.4]{Diss_Sass} for more discussion.

We introduce a discrete variant of the normal velocity $V=\bw\cdot\bn$, cf. \eqref{formnormalvel}. Almost everywhere on $S_h^n$, $n=1,\ldots,N$, we define 
\begin{equation}
	V_h\coloneqq \frac{\frac{-\partial \left(\hat{\phi}_h\circ(\Theta_h^n)^{-1}\right)}{\partial t}}{\norm{\nabla\left(\hat{\phi}_h\circ(\Theta_h^n)^{-1}\right)}}=\frac{-\left(\ddt{(\Theta_{h}^n)^{-1}}{t}\right)^{T}\left(\nabla_{(\x,t)}\hat{\phi}_h\circ(\Theta_h^n)^{-1}\right)}{\norm{(D\Theta_{h,\spa}^n)^{-T}\nabla\hat{\phi}_h}\circ(\Theta_h^n)^{-1}}.\label{Vh_def_ho}
\end{equation}
We will use this $V_h$ in Section~\ref{section_basic_geometry_est}. 
Its analogue on $S_{\lin}$ reads
\begin{equation}
	V_{\lin}\coloneqq \frac{-1}{\norm{\nabla\hat{\phi}_h}}\frac{\partial \hat{\phi}_h}{\partial t}.\label{Vh_def}
\end{equation}

\begin{remark} \label[remark]{remconsist}\rm
	The construction of $\Theta_h^n$ based on FE blending, cf. Remark~\ref{FEblend},  leads to surface discontinuities between time slabs (which is not the case for the smooth blending approach). We briefly explain the origin of these discontinuities, cf. \cite[Remark 5.3]{HeimannLehrenfeld}.
	Take a fixed $t=t_n$ and let $K \in \T$ be a tetrahedron that is cut by $\Gamma_{\rm lin}(t_n)$. By construction we then have $K \in \T_n^\Gamma \cap \T_{n+1}^\Gamma$. The sets of neighboring tetrahedra of $K$ in $ \T_n^\Gamma$ and in $\T^\Gamma_{n+1}$, however, are not necessarily the same. Due to an averaging at the finite element nodes used in the construction of  $\Theta_{h,t}^n$, cf.  \cite{lehrenfeld2016high,grande2018analysis,HeimannLehrenfeld},  the mappings $\Theta_{h,t_n}^{n}$ and $\Theta_{h,t_n}^{n+1}$ are in general not the same on $K$.
	Thus, for certain $\bx\in \Omega_n^{\Gamma}\cap \Omega^\Gamma_{n+1}$ we have
	\begin{equation*}
		\Theta_{h,s}^{n+1}(\bx,t_n)=\Theta_{h,t_n}^{n+1}(\bx)\neq\Theta_{h,t_n}^{n}(\bx)=\Theta_{h,s}^{n}(\bx,t_n), 
	\end{equation*}
	which in particular implies $\Gamma_h^{n+1}(t_n) \neq \Gamma^n(t_n)$. In a discretization method the transfer of information between time slabs has to take this possible surface discontinuity into account, cf. Section~\ref{secmethod}. 
\end{remark}

\subsection{Extension techniques}\label{ext_tech}

We briefly recall a standard closest point projection that we use to extend functions, that are  defined only  on $S$ or $S_h$, to a neighborhood.  We extend $\n$, cf.~\ref{A1},  spatially in a canonical way through the spatial signed distance function. We use the same \textit{spatial-only} extension for the \textit{space-time} normal $\n_S$. Let  $U$ be a sufficiently small neighborhood of $S$ such that the spatial signed distance function to $\Gamma(t)$, denoted by $\delta(\cdot,t) :U(t)  \rightarrow\mathbb{R}$ is $C^1$. The extension of $\bn$ is given by 
$
	\n(\x,t)\coloneqq \nabla \delta(\x,t)$, $(\bx,t) \in U$. 
 The (spatial) closest point projection is defined as
\begin{equation}
	\p:U\rightarrow S,\quad  \p(\x,t)\coloneqq (\x-\delta(\x,t)\mathbf{n}(\x,t),t).\label{p_nDEF}
\end{equation}
This  induces the unique decomposition
\begin{equation*}
	(\x,t)=\p(\x,t)+\delta(\x,t)(\n(\x,t),0), \quad (\x,t)\in U.
\end{equation*} 
It follows that  the restriction $\p\restrict{S_h^n}:S_h^n\rightarrow S^n$, $n=1,\ldots,N$, is bijective. 
For $n=1,\ldots,N$, we define $\p^{-1}_n:S^n\rightarrow S_h^n$, $\p^{-1}_n(\x,t)\coloneqq(\p\restrict{S_h^n})^{-1}(\x,t)$. For a  function $v:S\rightarrow \mathbb{R}$ we define its extension $v^e:U\rightarrow \mathbb{R}$ by
\begin{equation}
	v^e(\x,t)\coloneqq v(\p(\x,t)).\label{ext_def}
\end{equation}
Similarly, for a  function $v_h:S_h^n\rightarrow \mathbb{R}$ we define its extension $v_h^{l}:U^n:=U\restrict{\R^3 \times I_n}\rightarrow \mathbb{R}$ by
\begin{equation}
	v_h^{l}(\x,t)\coloneqq\begin{cases}v_h(\p^{-1}_n(\x,t)), \quad &(\x,t)\in S^n\\v_h^{l}(\p(\x,t)),\quad &(\x,t)\in  U^n \backslash S^n\end{cases}.\label{liftingdef}
\end{equation}
Both extensions of functions on $S$ and on $S_h^n$ are constant extensions in the direction of $\n$. Also note that for the extended normal $\bn= \nabla \delta$ we have  $\n=\n^e$.
\begin{remark}\label{st-extension} \rm
	In the analysis below we use extensions as in \eqref{ext_def} and \eqref{liftingdef} that are based on the \emph{spatial} closest point projection  \eqref{p_nDEF}. In our space-time setting it may seem (more) natural to use a space-time closest point projection, i.e., $\p_S:U\rightarrow S$ with $\p_S(\bx,t):={\rm argmin}_{\bz \in S}\|(\bx,t)-\bz\|$. An important difference to the case with the closed smooth surface $\Gamma(t)$ above is that the space-time surface $S$ and the time slab surfaces $S^n$ have boundaries. It is easy to see that due to this, in general the mappings $\p_S\restrict{S_h^1}$ and  $\p_S\restrict{S_h^N}$ (i.e. close to $\partial S$) are not injective. Also in general we have  ${\rm Range}\big(\p_S\restrict{S_h^n}) \neq  S^n$. In our analysis it is important (and natural) to use a bijection between $S^n$ and $S_h^n$ and time slab wise extensions.  The space-time closest point projection is clearly not very suitable for constructing such a bijection and extensions. For these reasons we use the construction based on the spatial closest point projection.  
\end{remark}

\section{Basic properties of the space-time surface approximation}\label{section_basic_geometry_est}
In order to analyze dicretization errors of  finite element methods that use the surface approximation $S_h$ one needs information about how certain surface characteristics on $S$ relate to corresponding ones on $S_h$. In this section we present such results. We derive a formula that relates the space-time surface measure on $S$ to that on $S_h$. We give a formula for the space-time normal on $S_h$ that can be directly compared to the formula \eqref{ns_schreibweise} for the space-time normal on $S$. Furthermore we present an analogue of the integral transformation formula \eqref{transformationsformelMAIN}  for the space-time surface approximation $S_h$ and give a partial integration formula on the space-time surface approximation. \\
We start with a comparison of surface measures on $S$ and $S_h$. Formulas that relate surface measures on $\Gamma(t)$ and $\Gamma_h(t)$ are known in the literature. We recall a result derived in \cite[Proposition 2.5]{demlow2009higher}. For any $t\in I_n$ let the two-dimensional surface measures of $\Gamma(t)$ and $\Gamma^n_h(t)$ be denoted by $\dif s$ and $\dif s_h$, respectively. On $U(t)$ we define $\kappa_i\coloneqq\frac{\kappa_i^e}{1+ \delta \kappa_i^e}$, where the principal curvatures $\kappa_1$, $\kappa_2$ and $0$ are the three eigenvalues of the Weingarten mapping $\mathbf{H}\coloneqq D \n$ on $U(t)$. On $\Gamma_h(t)$ let $\mu_h$ satisfy ${\mu_h\dif s_h=\dif s\circ \p}$. On $\Gamma_h(t)$ we then have 
	\begin{equation} \label{formelmuh}
		\mu_h=\mathbf{n}\cdot \mathbf{n}_{h}\prod_{i=1}^{2}(1-\delta\kappa_i).
\end{equation}

The  analysis in \cite[Proposition 2.5]{demlow2009higher}, which also applies to $d$-dimensional hypersurfaces with $d \neq 2$,  is \emph{not} applicable to the space-time surface $S$ and its approximation $S_h$. The reason for this is that in  our setting we want to use (only) the space-distance function $\delta(\cdot,t)$  to $\Gamma(t)$ and not a space-time distance function to $S$. The result in the following lemma yields a useful result. A proof is given in the Appendix.
\begin{lemma} \label[lemma]{lemmeasure}
	Let the three-dimensional surface measures of $S$ and $S_h$ be denoted by $\dif \sigma$ and $\dif \sigma_h$, respectively. Let  $\mu_h^S$ be such that ${\mu_h^S\dif \sigma_h=\dif \sigma\circ \p}$ on $S_h^n$. We define  $\n_0^T\coloneqq\left(\n^T,0\right)$. For $1 \leq n \leq N$ we have, with $\alpha$ as in \eqref{ns_schreibweise}, 
	\begin{align}
		\mu_h^S&=\alpha^e\, \n_0\cdot \mathbf{n}_{S_h}\prod_{i=1}^{2}(1-\delta\, \kappa_i) \quad \text{on}~~S_h^n. \label{formelmuhn}
	\end{align}
\end{lemma}

Note that the result in \eqref{formelmuhn} is general in the sense that it involves only geometric quantities related to $S_h$ and $S$. The way $S_h$ is constructed, e.g., in our case based on the space time mesh deformation $\Theta_h^n$, does not play a role. 
Based on the formula for $\mu_h^S$  in \cref{formelmuhn}  we now derive a  formula   that relates the space-time surface integral to a product integral.
\begin{lemma}\label{theo_trafo_discrete}
For $g \in L^2(S_h^n)$, $n=1, \ldots,N$,  we have
	\begin{equation} \label{transF1}
		\int_{t_{n-1}}^{t_n}\int_{\Gamma^n_h(t)}g \dif s_h\dif t=\int_{S_h^n} g \frac{\n_0\cdot \n_{S_h}}{\n\cdot \n_h}\dif \sigma_h, 
	\end{equation}
	with $\bn_0^T=(\bn^T, 0)$.
\end{lemma}
\begin{proof}
	We conclude using the continuous integral transformation \cref{transformationsformelMAIN} and the definitions of $\mu_h$, $\mu_h^S$ given in \cref{formelmuh} and \cref{formelmuhn}
	\begin{align}
		\begin{aligned}\label{calcu_fancy}
			&\int_{t_{n-1}}^{t_n}\int_{\Gamma^n_h(t)}g  \dif s_h\dif t=\int_{t_{n-1}}^{t_n}
			\int_{\Gamma(t)}g^l \frac{1}{\mu_h^l }\dif s\dif t
			=\int_{S^n}g^l \frac{1}{\mu_h^l \alpha }\dif \sigma\\
			&\qquad=\int_{S_h^n}g \frac{1}{\mu_h \alpha^e}\mu_h^S \dif \sigma_h
			=\int_{S_h^n}g \frac{\n_0\cdot \n_{S_h}}{\n \cdot \n_h}\dif \sigma_h.
		\end{aligned}
	\end{align}
\end{proof}

The result \eqref{transF1} is not satisfactory, yet, because it involves the normal vector $\bn$ of the exact surface. Below, in  \Cref{fancyformelcoro}, we derive a formula that is a direct discrete   analogue of \cref{transformationsformelMAIN} and uses only quantities available in the discretization method.  For this we need a discrete analogue of the relation \cref{ns_schreibweise}, derived in Lemma~\ref{lemma_nsh_repres} below. Note that 
\cref{ns_schreibweise} directly  follows from \eqref{A1} and the level set equation $\ddt{\phi}{t}+\w\cdot\nabla \phi =0$.
 We proceed analogously in the discrete setting. Due to the more complicated definitions of $\n_h$, $\n_{S_h}$ and $V_h$ the derivations are more technical. We start with auxiliary identities in the next lemma. The first two results \eqref{nsh_repres_lin}-\eqref{nabla_st_to_s} are  simple discrete analogues of relations for the exact level set function.  The third  identity \cref{vh_to_vlin} is basically a transformation formula between the ``linear'' normal velocity $V_{\lin}$ and its higher order version $V_h$. 
\begin{lemma}\label[lemma]{lemma_technical_vh}
	The following identities hold:
	\begin{alignat}{2}
        \n_{S_{\lin}} & =\frac{1}{\sqrt{1+V_{\lin}^2}}\begin{pmatrix}
			\n_{\lin}\\
		-V_{\lin} \end{pmatrix} && \text{a.e. on } S_{\lin} \label{nsh_repres_lin}\\
		\norm{\nabla_{(\x,t)}\hat{\phi}_h}&=\sqrt{1+V_{\lin}^2}\norm{\nabla\hat{\phi}_h}&& \text{a.e. on } S_{\lin},\label{nabla_st_to_s}\\
		\sqrt{1+V_h^2}&=\frac{\sqrt{1+V_{\lin}^2}\norm{(D_{(\x,t)}\Theta_h^n)^{-T}\n_{S_{\lin}}}}{\norm{(D\Theta_{h,\spa}^n)^{-T}\n_{\lin}}}\circ(\Theta_h^n)^{-1}\quad&& \text{a.e. on }S_h^n,~ 1 \leq n \leq N.\label{vh_to_vlin}
	\end{alignat}
\end{lemma}
\begin{proof}
 From  the definitions of $\n_{\lin}$, $\n_{S_{\lin}}$ and $V_{\lin}$ in \cref{nh_defs}, \cref{nslin_def_anddefintion_nsh_higherorder} and  \cref{Vh_def} we get (a.e. on $S_{\lin}$)
	\begin{equation*}
		\n_{S_{\lin}}=\frac{\nabla_{(\x,t)}\hat{\phi}_h}{\norm{\nabla_{(\x,t)}\hat{\phi}_h}}=\frac{\nabla_{(\x,t)}\hat{\phi}_h}{\sqrt{1+V_{\lin}^2}\norm{\nabla\hat{\phi}_h}}=\frac{1}{\sqrt{1+V_{\lin}^2}}\begin{pmatrix}
			\n_{\lin}\\-V_{\lin}
		\end{pmatrix}.
	\end{equation*}
	From this the results \eqref{nsh_repres_lin} and \eqref{nabla_st_to_s} immediately follow. We now consider the identity \eqref{vh_to_vlin}.
	For the squared denominator in this relation we introduce the notation $g_\Theta\coloneqq \norm{(D\Theta_{h,\spa}^n)^{-T}\n_{\lin}}^2\circ(\Theta_h^n)^{-1}$.
	For the squared  numerator we get, using the structure of $\Theta_h^n$ (cf. \eqref{Theta_Def}) and \eqref{nsh_repres_lin}
	\begin{align}
		\begin{aligned}
			&\left( (1+V_{\lin}^2)\norm{(D_{(\x,t)}\Theta_h^n)^{-T}\n_{S_{\lin}}}^2 \right)\circ(\Theta_h^n)^{-1} \\ & = (1+V_{\lin}^2\circ(\Theta_h^n)^{-1}) \norm{D^{T}_{(\x,t)}(\Theta_h^n)^{-1}\n_{S_{\lin}} \circ(\Theta_h^n)^{-1} }^2\\
			& = \big(1+V_{\lin}^2\circ(\Theta_h^n)^{-1}\big)
			\norm{D^{T}(\Theta_{h,\spa}^n)^{-1}\left(\frac{\n_{\lin}}{\sqrt{1+V_{\lin}^2}}\circ(\Theta_h^n)^{-1}\right)}^2 \\ & \quad + \big(1+V_{\lin}^2\circ(\Theta_h^n)^{-1}\big)  \norm{\left(\ddt{(\Theta_{h}^n)^{-1}}{t}\right)^{T}\left(\n_{S_{\lin}}\circ(\Theta_h^n)^{-1}\right)}^2 \\
			& \quad = g_\Theta +  \norm{\left(\ddt{(\Theta_{h}^n)^{-1}}{t}\right)^{T}\left(\frac{\nabla_{(\bx,t)}\hat{\phi}_h}{\|\nabla \hat{\phi}_h\|} \circ(\Theta_h^n)^{-1}\right)}^2
			\label{proof_vh_to_vlin_1},
		\end{aligned}
	\end{align}
	where in the last equality we used the relation $\sqrt{1+ V_{\rm lin}^2} \|\nabla \hat\phi_h\| \n_{S_{\lin}} = \nabla_{(\bx,t)}\hat{\phi}_h $. 
	Hence, using $ \|\hat \phi_h\|\n_{\lin} = \nabla \hat \phi_h$ we conclude that it remains to verify the identity
	\begin{equation} \label{66A}
	 \frac{\norm{\left(\ddt{(\Theta_{h}^n)^{-1}}{t}\right)^{T}\left(\nabla_{(\bx,t)}\hat{\phi}_h \circ(\Theta_h^n)^{-1}\right)} }{
	\norm{(D\Theta_{h,\spa}^n)^{-T} \nabla \hat \phi_h}\circ(\Theta_h^n)^{-1}} = |V_h| :=
\frac{ |\ddt{\left(\hat{\phi}_h\circ(\Theta_h^n)^{-1}\right)}{t}|}{\norm{\nabla\left(\hat{\phi}_h\circ(\Theta_h^n)^{-1}\right)}}.
	  \end{equation}
The numerators in \eqref{66A} agree, which follows from the chain rule.  The denominators also agree, which again follows from the chain rule and the fact that the spatial derivative of the last component of $\Theta_h^n$ vanishes, cf. \cref{Theta_Def}.
\end{proof}

 We use \Cref{lemma_technical_vh} in the proof of the following lemma that relates the (higher order) discrete spatial normal $\n_h$ and the (higher order) discrete space-time normal $\n_{S_h}$ in a formula with the same structure as \eqref{ns_schreibweise}.
\begin{lemma}\label[lemma]{lemma_nsh_repres}
	Almost everywhere on $S_h^n$, $n=1,\ldots,N$, we have the identity
	\begin{equation}
		\n_{S_h}=\frac{1}{\sqrt{1+V_h^2}}\begin{pmatrix}
			\n_h\\-V_h
		\end{pmatrix}.\label{nsh_repres}
	\end{equation}
\end{lemma}
\begin{proof}
Recall the formulas \eqref{nh_defs} and \eqref{Vh_def_ho}
\[ \bn_h \circ \Theta_h^n= \frac{(D\Theta_{h,\spa}^n)^{-T}\bn_{\lin}}{\norm{(D\Theta_{h,\spa}^n)^{-T}\bn_{\lin}}}, \quad 
V_h\circ \Theta_h^n=\frac{-\left(\ddt{(\Theta_{h}^n)^{-1}}{t} \circ\Theta_h^n \right)^{T}\nabla_{(\x,t)}\hat{\phi}_h}{\norm{(D\Theta_{h,\spa}^n)^{-T}\nabla\hat{\phi}_h}}.
\]
 Using \eqref{nslin_def_anddefintion_nsh_higherorder} and  the structure of $ \Theta_h^n$, cf. \eqref{Theta_Def} we obtain
	\begin{equation} \label{TT}
		\n_{S_{h}} 
    \circ \Theta_{h}^n= \frac{1}{\norm{(D_{(\x,t)}\Theta_h^n)^{-T}\n_{S_{\lin}}}}\begin{pmatrix} (D\Theta_{h,\spa}^n)^{-T} \left(\frac{\n_{\lin}}{\sqrt{1+V_{\lin}^2}}\right)\\ \left(\ddt{(\Theta_h^n)^{-1}}{t}\circ \Theta_h^n\right)^{T}\n_{S_{\lin}}\end{pmatrix}.
	\end{equation}
	From \eqref{vh_to_vlin} we get
	\[
   \norm{(D_{(\x,t)}\Theta_h^n)^{-T}\n_{S_{\lin}}}^{-1}= \frac{\sqrt{1+V_{\lin}^2}}{\sqrt{ 1+ V_h^2\circ\Theta_h^n}} \norm{(D\Theta_{h,\spa}^n)^{-T}\n_{\lin}}^{-1}.
	\]
Hence, for the first three components of the vector $\n_{S_{h}}\circ \Theta_h^n \in\R^4$ we get
\[
  \frac{(D\Theta_{h,\spa}^n)^{-T} \left(\frac{\n_{\lin}}{\sqrt{1+V_{\lin}^2}}\right)}{\norm{(D_{(\x,t)}\Theta_h^n)^{-T}\n_{S_{\lin}}}} = \frac{1}{\sqrt{ 1+ V_h^2\circ\Theta_h^n}}\bn_h \circ \Theta_h^n,
\]
and for the last entry we get, using \eqref{nslin_def_anddefintion_nsh_higherorder}, \eqref{nabla_st_to_s} and $\norm{\nabla \hat \phi_h}\n_{\lin}= \nabla \hat \phi_h $, 
\[
 \frac{\left(\ddt{(\Theta_h^n)^{-1}}{t}\circ \Theta_h^n\right)^{T}\n_{S_{\lin}}}{\norm{(D_{(\x,t)}\Theta_h^n)^{-T}\n_{S_{\lin}}}}= \frac{\left(\ddt{(\Theta_h^n)^{-1}}{t}\circ \Theta_h^n\right)^{T} \nabla_{(\bx,t)}\hat \phi_h}{\sqrt{1+ V_h^2 \circ \Theta_h^n} \norm{\nabla \hat \phi_h} \norm{(D\Theta_{h,\spa}^n)^{-T}\n_{\lin} }}
 = \frac{-V_h \circ \Theta_h^n}{\sqrt{1+ V_h^2 \circ \Theta_h^n}},
\]
which completes the proof.
\end{proof}

As an easy consequence we obtain a discrete analogue of the integral transformation  formula.
\begin{corollary}\label[corollary]{fancyformelcoro}
	For $g_h\in L^2(S_h^n)$,  $n=1,\dots,N$, we have
	\begin{equation} \label{PPPP}
	\int_{t_{n-1}}^{t_n}\int_{\Gamma^n_h(t)}g_h\dif s_h\dif t=\int_{S_h^n}\frac{g_h}{\sqrt{1+V_h^2}}\dif \sigma_h.
\end{equation}
\end{corollary}
\begin{proof}
The result follows from Lemma~\ref{theo_trafo_discrete}, \Cref{lemma_nsh_repres} and $\bn_0 \cdot \bn_{S_h} = (1 + V_h^2)^{- \frac12} \bn\cdot \bn_h$.
\end{proof}
\begin{remark}
 \rm We briefly comment on the structure of the proof of \eqref{PPPP}. A proof of \cref{transformationsformelMAIN} can be given using  local parametrizations of $\Gamma(t)$ and of $S$, which can be related using the normal velocity $V=\bw \cdot \bn$ that transports $\Gamma(t)$. This approach does not work for the discrete version  \eqref{PPPP}, because we  do not have a  velocity field available that transports $\Gamma_h^n(t)$. Instead, in our proof we relate integrals $\int_{\Gamma_h^n(t)} \sim \int_{\Gamma(t)}$ and $\int_{S_h^n} \sim \int_{S^n}$, and use the result \cref{transformationsformelMAIN}.
\end{remark}
\ \\[1ex]
Using \Cref{fancyformelcoro} we are able to transform product integrals on $\Gamma_h^n(t)\times I_n$ to surface integrals on $S_h^n$. For the latter  we have a natural  partial integration formula, which   is  presented below in Theorem~\ref{PI_Theorem}. A further useful relation between the surface measures on $S$ and $S_h$, which follows from the results above, is given by 
\begin{equation} \label{helptrans}
 \frac{1}{\alpha} \dif \sigma = \frac{\mu_h}{\sqrt{1 +V_h^2}} \dif \sigma_h.
\end{equation}

\subsection{Partial integration formula on the space-time surface approximation}
Below we use standard surface differential operators on the  smooth  surfaces $\Gamma= \Gamma(t)$ (for fixed $t$) and $S$, denoted by $\nabla_\gamma$ (surface gradient) and $\Div_{\gamma}$, $\gamma \in \{ \Gamma, S\}$. Corresponding to the velocity field $\bw$ we define the space-time 
vector $\bw_S^T\coloneqq(\bw^T,1) \in \R^4$. Elementary calculations, cf. \cite[Lemma 3.2]{Diss_Sass}, yield the relation
\begin{equation} \label{elneu}
  {\Div_S \left( \frac{1}{\alpha} \bw_S \right)=\frac{1}{\alpha} \Div_{\Gamma} \bw} \quad \text{on}~S.
\end{equation}
For a scalar function $v \in H^1(S)$ the material derivative is given by
\begin{equation}
  \dot v:= \bw_S \cdot \nabla_S v = \frac{\partial v^e}{\partial t} + \bw \cdot \nabla v^e\label{def_weakdev_h1}.
\end{equation}
In our setting we are interested in partial integration for product integrals of the form
$\int_0^T \int_{\Gamma(t)} g \dif s\dif t = \int_S \frac{1}{\alpha} g \dif \sigma$, with $\alpha=\sqrt{1+(\bw \cdot \bn)^2}$. The following partial integration formula holds
\begin{equation} \label{PIcont}
\begin{split}
 \int_{S}\frac{1}{\alpha}\dot{u}v\dif \sigma & =-\int_{S}\frac{1}{\alpha}u\dot{v}\dif \sigma +\int_S u v  \Div_S \left( \frac{1}{\alpha} \bw_S \right) \dif \sigma  \\
 & \quad + \int_{\Gamma(T)} uv \dif s - \int_{\Gamma(0)} uv \dif s.
\end{split} \end{equation}
This result can be derived using the Leibniz formula and the relation \eqref{elneu}. In the setting of a finite element discretization on the space-time surface approximation $S_h$ we are interested in a discrete analogue of \eqref{PIcont}. 

Note that the space-time surface $S_h^n$ is a Lipschitz surface formed by a triangulation of 
smooth curved space-time surfaces, cf. \Cref{theta_picture} for  the case of an evolving curve embedded in $\R^2$. Let the space-time surface triangulation $\mathcal{T}_{S_h^n}$ be the set of smooth three-dimensional manifolds that $S_{h}^n$ consists of, i.e. 
\[ {\T_{S_h^n}\coloneqq \{\Theta_h^n(P)\cap S_h^n: P\in Q_{h,n}^S\}}.
\]
Then, we have ${S_h^n=\bigcup_{K_S\in \mathcal{T}_{S_h^n}} K_S}$. Let $\T_{S_h}\coloneqq \bigcup_{n=1}^N\T_{S_h^n}$. Note that $K_S\in \T_{S_h}$ is a three-dimensional curved polytope that can be partitioned into curved tetrahedra. In general this deformed triangulation $\T_{S_h}$ is \emph{not} shape regular; due to the evolving surface through the fixed bulk triangulation  $Q_h$, element angles can be very small and the size of neighboring elements may vary strongly.  An illustration of the two-dimensional analogue of $\T_{S_h^n}$ is given in \Cref{theta_picture}.
 For $n= 1,\dots,N,$ let $\mathcal{F}_I^n\coloneqq \{ \partial K_S^1\cap\partial K_S^2:K_S^1,K_S^2\in \mathcal{T}_{S_h^n},K_S^1\neq K_S^2\}$ be the set of interior curved boundary faces of the elements in $\mathcal{T}_{S_h^n}$ and let $\mathcal{F}_I\coloneqq \bigcup_{n=1}^N\mathcal{F}_I^n$.
Almost everywhere on $\partial S_h^n$ the vector  
\begin{equation}
	\bnu_{\partial}\coloneqq \frac{1}{\sqrt{1+V_h^2}} \begin{pmatrix} V_h \bn_h \\ 1 \end{pmatrix} \label{ndelta_defs}
\end{equation}
 is orthogonal to $\partial S_h^n$ and satisfies $\bnu_{\partial}\cdot  \n_{S_h}=0$, cf.~\eqref{nsh_repres}. Hence, $\bnu_{\partial}$ is the unit conormal of the top boundary of $S_h^n$ and $-\bnu_{\partial }$ is the unit outer conormal of the bottom boundary of $S_h^n$. For $K_S\in \T_{S_h}$ the interior conormals are denoted by ${\bnu_h}\restrict{K_S}$, 
 cf. \Cref{conormal} for an illustration.
 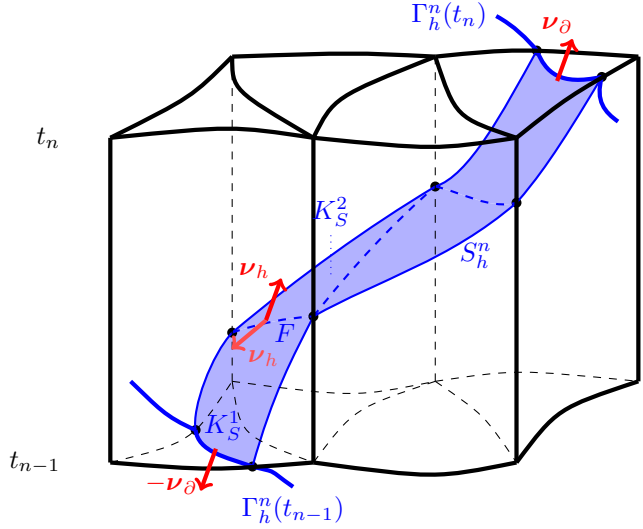
\begin{figure}[!htbp]
	\centering
	\begin{tikzpicture}[scale = 0.54]
		\newcommand{\z}{0.05}
		\newcommand{\Hoehe}{8.0}
		\newcommand{\Tiefe}{2.0}
		\newcommand{\Shift}{3.0}
		
		\newcommand{\lambdaA}{0.5}
		\newcommand{\PunktA}{(\lambdaA*0+\Shift-\lambdaA*\Shift+0.6,-0.2+\lambdaA*0+\Tiefe-\lambdaA*\Tiefe)}
		\newcommand{\lambdaB}{0.7}
		\newcommand{\PunktB}{(\lambdaB*5,0-0.1)}
		\newcommand{\lambdaC}{0.15}
		\newcommand{\PunktC}{(\Shift,\Tiefe+\lambdaC*\Hoehe)}
		\newcommand{\lambdaD}{0.45}
		\newcommand{\PunktD}{(5,\lambdaD*\Hoehe)}
		\newcommand{\lambdaE}{0.6}
		\newcommand{\PunktE}{(5+\Shift,\Tiefe+\lambdaE*\Hoehe)}
		\newcommand{\lambdaF}{0.8}
		\newcommand{\PunktF}{(10,\lambdaF*\Hoehe)}
		\newcommand{\lambdaG}{0.5}
		\newcommand{\PunktG}{(\lambdaG*5+\lambdaG*\Shift+10+\Shift-10*\lambdaG-\lambdaG*\Shift,\Tiefe+\Hoehe-0.35+0.5)}
		\newcommand{\lambdaH}{0.5}
		\newcommand{\PunktH}{(\lambdaH*10+10+\Shift-\lambdaH*10-\lambdaH*\Shift+0.6,\Tiefe-\lambdaH*\Tiefe+\Hoehe+0.5)}
		\newcommand{\AtildeShifth}{-0.5}\newcommand{\AtildeShiftv}{1.0}
		\newcommand{\PunktAtilde}{(\lambdaA*0+\Shift-\lambdaA*\Shift+\AtildeShifth-0.5,\lambdaA*0+\Tiefe-\lambdaA*\Tiefe+\AtildeShiftv)}
		\newcommand{\BtildeShifth}{1.0}\newcommand{\BtildeShiftv}{-0.4}
		\newcommand{\PunktBtilde}{(\lambdaB*5+\BtildeShifth,-0.2+\BtildeShiftv)}
		\newcommand{\GtildeShifth}{-1.0}\newcommand{\GtildeShiftv}{+1.0}
		\newcommand{\PunktGtilde}{(\lambdaG*5+\lambdaG*\Shift+10+\Shift-10*\lambdaG-\lambdaG*\Shift+\GtildeShifth,\Tiefe+\Hoehe+\GtildeShiftv)}
		\newcommand{\HtildeShifth}{1.0}\newcommand{\HtildeShiftv}{-0.6}
		\newcommand{\PunktHtilde}{(\lambdaH*10+10+\Shift-\lambdaH*10-\lambdaH*\Shift+\HtildeShifth,\Tiefe-\lambdaH*\Tiefe+\Hoehe+\HtildeShiftv)}
		
		\newcommand{\ABh}{0.0}\newcommand{\ABv}{-0.2}
		\newcommand{\BAh}{0.0}\newcommand{\BAv}{-0.4}
		\newcommand{\BDh}{0.0}\newcommand{\BDv}{0.2}
		\newcommand{\DBh}{0.0}\newcommand{\DBv}{0.5}
		\newcommand{\ACh}{0.0}\newcommand{\ACv}{0.2}
		\newcommand{\CAh}{0.0}\newcommand{\CAv}{-0.4}
		\newcommand{\CEh}{0.0}\newcommand{\CEv}{0.2}
		\newcommand{\ECh}{0.0}\newcommand{\ECv}{0.2}
		\newcommand{\DFh}{0.0}\newcommand{\DFv}{-0.5}
		\newcommand{\FDh}{0.0}\newcommand{\FDv}{0.0}
		\newcommand{\CDh}{0.0}\newcommand{\CDv}{0.1}
		\newcommand{\DCh}{0.0}\newcommand{\DCv}{0.1}
		\newcommand{\DEh}{0.0}\newcommand{\DEv}{0.2}
		\newcommand{\EDh}{0.0}\newcommand{\EDv}{0.2}
		\newcommand{\EFh}{0.0}\newcommand{\EFv}{-0.3}
		\newcommand{\FEh}{0.0}\newcommand{\FEv}{0.0}
		\newcommand{\EGh}{0.0}\newcommand{\EGv}{-0.3}
		\newcommand{\GEh}{0.0}\newcommand{\GEv}{-0.8}
		\newcommand{\FHh}{0.0}\newcommand{\FHv}{-0.5}
		\newcommand{\HFh}{0.0}\newcommand{\HFv}{-0.3}
		\newcommand{\GHh}{-0.1}\newcommand{\GHv}{-0.05}
		\newcommand{\HGh}{-0.1}\newcommand{\HGv}{-0.05}
		\newcommand{\AAtildeh}{0.0}\newcommand{\AAtildev}{-0.3}
		\newcommand{\AtildeAh}{0.0}\newcommand{\AtildeAv}{-0.2}
		\newcommand{\BBtildeh}{0.0}\newcommand{\BBtildev}{0.0}
		\newcommand{\BtildeBh}{0.0}\newcommand{\BtildeBv}{0.0}
		\newcommand{\HHtildeh}{0.01}\newcommand{\HHtildev}{0.03}
		\newcommand{\HtildeHh}{0.1}\newcommand{\HtildeHv}{0.1}
		\newcommand{\GGtildeh}{0.1}\newcommand{\GGtildev}{0.1}
		\newcommand{\GtildeGh}{0.1}\newcommand{\GtildeGv}{0.1}
		
		\newcommand{\PunktAB}{(0.333*\lambdaA*0+0.333*\Shift-0.333*\lambdaA*\Shift+0.666*\lambdaB*5+\ABh,0.333*\lambdaA*0+0.333*\Tiefe-0.333*\lambdaA*\Tiefe+0.666*0+\ABv)}
		\newcommand{\PunktBA}{(0.666*\lambdaA*0+0.666*\Shift-0.666*\lambdaA*\Shift+0.333*\lambdaB*5+\BAh,0.666*\lambdaA*0+0.666*\Tiefe-0.666*\lambdaA*\Tiefe+0.333*0+\BAv)}
		\newcommand{\PunktBD}{(0.333*\lambdaB*5+0.666*5+\BDh,0.333*0+0.666*\lambdaD*\Hoehe+\BDv)}
		\newcommand{\PunktDB}{(0.666*\lambdaB*5+0.333*5+\DBh,0.666*0+0.333*\lambdaD*\Hoehe+\DBv)}
		\newcommand{\PunktAC}{(0.333*\lambdaA*0+0.333*\Shift-0.333*\lambdaA*\Shift+0.666*\Shift+\ACh,0.333*\lambdaA*0+0.333*\Tiefe-0.333*\lambdaA*\Tiefe+0.666*\Tiefe+0.666*\lambdaC*\Hoehe+\ACv)}
		\newcommand{\PunktCA}{(0.666*\lambdaA*0+0.666*\Shift-0.666*\lambdaA*\Shift+0.333*\Shift+\CAh,0.666*\lambdaA*0+0.666*\Tiefe-0.666*\lambdaA*\Tiefe+0.333*\Tiefe+0.333*\lambdaC*\Hoehe+\CAv)}
		\newcommand{\PunktCE}{(0.333*\Shift+0.666*5+0.666*\Shift+\CEh,0.333*\Tiefe+0.333*\lambdaC*\Hoehe+0.666*\Tiefe+0.666*\lambdaE*\Hoehe+\CEv)}
		\newcommand{\PunktEC}{(0.666*\Shift+0.333*5+0.333*\Shift+\ECh,0.666*\Tiefe+0.666*\lambdaC*\Hoehe+0.333*\Tiefe+0.333*\lambdaE*\Hoehe+\ECv)}
		\newcommand{\PunktDF}{(0.333*5+0.666*10+\DFh,0.333*\lambdaD*\Hoehe+0.666*\lambdaF*\Hoehe+\DFv)}
		\newcommand{\PunktFD}{(0.666*5+0.333*10+\FDh,0.666*\lambdaD*\Hoehe+0.333*\lambdaF*\Hoehe+\FDv)}
		\newcommand{\PunktCD}{(0.333*\Shift+0.666*5+\CDh,0.333*\Tiefe+0.333*\lambdaC*\Hoehe+0.666*\lambdaD*\Hoehe+\CDv)}
		\newcommand{\PunktDC}{(0.666*\Shift+0.333*5+\DCh,0.666*\Tiefe+0.666*\lambdaC*\Hoehe+0.333*\lambdaD*\Hoehe+\DCv)}
		\newcommand{\PunktDE}{(0.666*5+0.666*\Shift+0.333*5+\DEh,0.666*\Tiefe+0.666*\lambdaE*\Hoehe+0.333*\lambdaD*\Hoehe+\DEv)}
		\newcommand{\PunktED}{(0.333*5+0.333*\Shift+0.666*5+\EDh,0.333*\Tiefe+0.333*\lambdaE*\Hoehe+0.666*\lambdaD*\Hoehe+\EDv)}
		\newcommand{\PunktEF}{(0.333*5+0.333*\Shift+0.666*10+\EFh,0.333*\Tiefe+0.333*\lambdaE*\Hoehe+0.666*\lambdaF*\Hoehe+\EFv)}
		\newcommand{\PunktFE}{(0.666*5+0.666*\Shift+0.333*10+\FEh,0.666*\Tiefe+0.666*\lambdaE*\Hoehe+0.333*\lambdaF*\Hoehe+\FEv)}
		\newcommand{\PunktEG}{(0.333*5+0.333*\Shift+0.666*\lambdaG*5+0.666*\lambdaG*\Shift+0.666*10+0.666*\Shift-0.666*10*\lambdaG-0.666*\lambdaG*\Shift+\EGh,0.333*\Tiefe+0.333*\lambdaE*\Hoehe+0.666*\Tiefe+0.666*\Hoehe+\EGv)}
		\newcommand{\PunktGE}{(0.666*5+0.666*\Shift+0.333*\lambdaG*5+0.333*\lambdaG*\Shift+0.333*10+0.333*\Shift-0.333*10*\lambdaG-0.333*\lambdaG*\Shift+\GEh,0.666*\Tiefe+0.666*\lambdaE*\Hoehe+0.333*\Tiefe+0.333*\Hoehe+\GEv)}
		\newcommand{\PunktFH}{(0.333*10+0.666*\lambdaH*10+0.666*10+0.666*\Shift-0.666*\lambdaH*10-0.666*\lambdaH*\Shift+\FHh,0.333*\lambdaF*\Hoehe+0.666*\Tiefe-0.666*\lambdaH*\Tiefe+0.666*\Hoehe+\FHv)}
		\newcommand{\PunktHF}{(0.666*10+0.333*\lambdaH*10+0.333*10+0.333*\Shift-0.333*\lambdaH*10-0.333*\lambdaH*\Shift+\HFh,0.666*\lambdaF*\Hoehe+0.333*\Tiefe-0.333*\lambdaH*\Tiefe+0.333*\Hoehe+\HFv)}
		\newcommand{\PunktGH}{(0.333*\lambdaG*5+0.333*\lambdaG*\Shift+0.333*10+0.333*\Shift-0.333*10*\lambdaG-0.333*\lambdaG*\Shift+0.666*\lambdaH*10+0.666*10+0.666*\Shift-0.666*\lambdaH*10-0.666*\lambdaH*\Shift+\GHh,0.333*\Tiefe+0.333*\Hoehe+0.666*\Tiefe-0.666*\lambdaH*\Tiefe+0.666*\Hoehe+\GHv)}
		\newcommand{\PunktHG}{(0.666*\lambdaG*5+0.666*\lambdaG*\Shift+0.666*10+0.666*\Shift-0.666*10*\lambdaG-0.666*\lambdaG*\Shift+0.333*\lambdaH*10+0.333*10+0.333*\Shift-0.333*\lambdaH*10-0.333*\lambdaH*\Shift+\HGh,0.666*\Tiefe+0.666*\Hoehe+0.333*\Tiefe-0.333*\lambdaH*\Tiefe+0.333*\Hoehe+\HGv)}
		\newcommand{\PunktAAtilde}{(\lambdaA*0+\Shift-\lambdaA*\Shift+0.666*\AtildeShifth+\AAtildeh,\lambdaA*0+\Tiefe-\lambdaA*\Tiefe+0.666*\AtildeShiftv+\AAtildev)}
		\newcommand{\PunktAtildeA}{(\lambdaA*0+\Shift-\lambdaA*\Shift+0.333*\AtildeShifth+\AtildeAh,\lambdaA*0+\Tiefe-\lambdaA*\Tiefe+0.333*\AtildeShiftv+\AtildeAv)}
		\newcommand{\PunktBBtilde}{(\lambdaB*5+0.333*\BtildeShifth+\BBtildeh,0.333*\BtildeShiftv+\BBtildev)}
		\newcommand{\PunktBtildeB}{(\lambdaB*5+0.666*\BtildeShifth+\BtildeBh,0.666*\BtildeShiftv+\BtildeBv)}
		\newcommand{\PunktHHtilde}{(\lambdaH*10+10+\Shift-\lambdaH*10-\lambdaH*\Shift+0.666*\HtildeShifth+\HHtildeh,\Tiefe-\lambdaH*\Tiefe+\Hoehe+0.666*\HtildeShiftv+\HHtildev)}
		\newcommand{\PunktHtildeH}{(\lambdaH*10+10+\Shift-\lambdaH*10-\lambdaH*\Shift+0.333*\HtildeShifth+\HtildeHh,\Tiefe-\lambdaH*\Tiefe+\Hoehe+0.333*\HtildeShiftv+\HtildeHv)}
		\newcommand{\PunktGGtilde}{(\lambdaG*5+\lambdaG*\Shift+10+\Shift-10*\lambdaG-\lambdaG*\Shift+0.666*\GtildeShifth+\GGtildeh,\Tiefe+\Hoehe+0.666*\GtildeShiftv+\GGtildev)}
		\newcommand{\PunktGtildeG}{(\lambdaG*5+\lambdaG*\Shift+10+\Shift-10*\lambdaG-\lambdaG*\Shift+0.333*\GtildeShifth+\GtildeGh,\Tiefe+\Hoehe+0.333*\GtildeShiftv+\GtildeGv)}
		
		\node[left] at (-1.0,0.0) {$t_{n-1}$};
		\node[left] at (-1.0,\Hoehe) {$t_n$};
		
		\newcommand{\FdM}{blue}
		
		
		\node[below, color=\FdM] at \PunktBtilde {$\Gamma_h^n(t_{n-1})$};
		\node[left, color=\FdM] at \PunktGtilde {$\Gamma_h^n(t_n)$};
		
		\newcommand{\FF}{30}
		\fill[\FdM!\FF] \PunktA .. controls \PunktBA and \PunktAB .. \PunktB -- \PunktB .. controls \PunktDB and \PunktBD .. \PunktD -- %
		\PunktD .. controls \PunktCD and \PunktDC .. \PunktC -- \PunktC .. controls \PunktAC and \PunktCA .. \PunktA;
		\fill[\FdM!\FF] \PunktC .. controls \PunktDC and \PunktCD .. \PunktD -- \PunktD .. controls \PunktFD and \PunktDF .. \PunktF -- %
		\PunktF .. controls \PunktEF and \PunktFE .. \PunktE -- \PunktE .. controls \PunktCE and \PunktEC .. \PunktC;
		\fill[\FdM!\FF] \PunktE .. controls \PunktFE and \PunktEF .. \PunktF -- \PunktF .. controls \PunktHF and \PunktFH .. \PunktH -- %
		\PunktH .. controls \PunktGH and \PunktHG .. \PunktG -- \PunktG .. controls \PunktEG and \PunktGE .. \PunktE;

		\newcommand{\Radius}{2pt}
		\filldraw[fill=black, ultra thick] \PunktA circle (\Radius);
		\filldraw[fill=black, ultra thick] \PunktB circle (\Radius);
		\filldraw[fill=black, ultra thick] \PunktC circle (\Radius);
		\filldraw[fill=black, ultra thick] \PunktD circle (\Radius);
		\filldraw[fill=black, ultra thick] \PunktE circle (\Radius);
		\filldraw[fill=black, ultra thick] \PunktF circle (\Radius);
		\filldraw[fill=black, ultra thick] \PunktG circle (\Radius);
		\filldraw[fill=black, ultra thick] \PunktH circle (\Radius);
		
		\draw[ultra thick, color=\FdM] \PunktA .. controls \PunktBA and \PunktAB .. \PunktB;
		\draw[thick, color=\FdM]       \PunktA .. controls \PunktCA and \PunktAC .. \PunktC;
		\draw[thick, color=\FdM]       \PunktC .. controls \PunktEC and \PunktCE .. \PunktE;
		\draw[thick, color=\FdM]       \PunktE .. controls \PunktGE and \PunktEG .. \PunktG;
		\draw[thick, color=\FdM]       \PunktB .. controls \PunktDB and \PunktBD .. \PunktD;
		\draw[thick, color=\FdM]       \PunktD .. controls \PunktFD and \PunktDF .. \PunktF;
		\draw[thick, color=\FdM]       \PunktF .. controls \PunktHF and \PunktFH .. \PunktH;
		\draw[ultra thick, color=\FdM] \PunktH .. controls \PunktGH and \PunktHG .. \PunktG;
		\draw[ultra thick, color=\FdM] \PunktA .. controls \PunktAtildeA and \PunktAAtilde .. \PunktAtilde;
		\draw[ultra thick, color=\FdM] \PunktB .. controls \PunktBtildeB and \PunktBBtilde .. \PunktBtilde;
		\draw[ultra thick, color=\FdM] \PunktG .. controls \PunktGtildeG and \PunktGGtilde .. \PunktGtilde;
		\draw[ultra thick, color=\FdM] \PunktH .. controls \PunktHtildeH and \PunktHHtilde .. \PunktHtilde;
		
		\draw[thick, dashed, color=\FdM]       \PunktC .. controls \PunktDC and \PunktCD .. \PunktD;
		\draw[thick, dashed, color=\FdM]       \PunktD .. controls \PunktED and \PunktDE .. \PunktE;
		\draw[thick, dashed, color=\FdM]       \PunktE .. controls \PunktFE and \PunktEF .. \PunktF;
		
		\draw[ultra thick] (0.0,0.0) .. controls(2.5, -0.2) .. (5.0,0.0);
		\draw[ultra thick] (5.0,0.0) .. controls (8, -0.4) .. (10.0,0.0);
		\draw[dashed] (\Shift,\Tiefe) .. controls (3+\Shift, \Tiefe-0.4) .. (5.0+\Shift,\Tiefe);
		\draw[dashed] (5+\Shift,\Tiefe) .. controls (7+\Shift, \Tiefe+0.3) .. (10.0+\Shift,\Tiefe);
		\draw[dashed] (0.0,0.0) .. controls(\Shift/1.5, \Tiefe/4) .. (\Shift,\Tiefe);
		\draw[dashed] (\Shift,\Tiefe) .. controls(\Shift+0.2, \Tiefe/3) .. (5.0,0.0);
		\draw[dashed] (5.0,0.0) .. controls(6, \Tiefe/1.5) .. (5.0+\Shift,\Tiefe);
		\draw[dashed] (5.0+\Shift,\Tiefe) .. controls (7.7+\Shift/2, \Tiefe/1.7) .. (10.0,0.0);
		\draw[ultra thick] (10.0,0.0) .. controls (10 + \Shift/3, \Tiefe/1.5) .. (10.0+\Shift,\Tiefe);
		
		\draw[ultra thick] (0.0,0.0+\Hoehe) .. controls (3, 0.3+\Hoehe) .. (5.0,0.0+\Hoehe);
		\draw[ultra thick] (5.0,0.0+\Hoehe) .. controls (8, -0.3+\Hoehe) .. (10.0,0.0+\Hoehe);
		\draw[ultra thick] (\Shift,\Tiefe+\Hoehe) .. controls (3+\Shift, \Tiefe+\Hoehe+0.1) .. (5.0+\Shift,\Tiefe+\Hoehe);
		\draw[ultra thick] (\Shift +5 ,\Tiefe+\Hoehe) .. controls (7+\Shift, \Tiefe+\Hoehe+0.2) .. (10.0+\Shift,\Tiefe+\Hoehe);
		\draw[ultra thick] (0.0,0.0+\Hoehe) .. controls (\Shift/1.5, \Tiefe/2+\Hoehe) .. (\Shift,\Tiefe+\Hoehe);
		\draw[ultra thick] (\Shift,\Tiefe+\Hoehe) .. controls (\Shift/2+1.5, \Hoehe+\Tiefe/1.7) .. (5.0,0.0+\Hoehe);
		\draw[ultra thick] (5.0,0.0+\Hoehe) .. controls (4+ \Shift/2, \Tiefe/2+\Hoehe) .. (5.0+\Shift,\Tiefe+\Hoehe);
		\draw[ultra thick] (5.0+\Shift,\Tiefe+\Hoehe) .. controls(7.5+\Shift/3, \Tiefe/4 + \Hoehe) .. (10.0,0.0+\Hoehe);
		\draw[ultra thick] (10.0,0.0+\Hoehe) .. controls(10+\Shift/2.4, \Tiefe/2+\Hoehe) .. (10.0+\Shift,\Tiefe+\Hoehe);
		
		\draw[ultra thick] (0.0,0.0) -- (0.0,0.0+\Hoehe);
		\draw[ultra thick] (5.0,0.0) -- (5.0,0.0+\Hoehe);
		\draw[ultra thick] (10.0,0.0) -- (10.0,0.0+\Hoehe);
		\draw[dashed]      (\Shift,\Tiefe) -- (\Shift,\Tiefe+\Hoehe);
		\draw[dashed]      (5.0+\Shift,\Tiefe) -- (5.0+\Shift,\Tiefe+\Hoehe);
		\draw[ultra thick] (10.0+\Shift,\Tiefe) -- (10.0+\Shift,\Tiefe+\Hoehe);
		\node[above, color=\FdM] at (3*\Shift,\Hoehe*0.57) {$S_{h}^n$};
		\node[above, color=\FdM] at (0.925*\Shift,\Hoehe*0.04) {$K_S^1$};
		\node[above, color=\FdM] at (1.81*\Shift,\Hoehe*0.7) {$K_S^2$};
		\draw[color=blue, dotted ] (1.81*\Shift,\Hoehe*0.7) -- (1.81*\Shift,\Hoehe*0.575);
		\node[above, color=\FdM] at (4.3,2.75) {$F$};
		\draw[->, color=red, ultra thick] (2.6, 0.34) -- ++(250:1.1);
		\node[above, color=red] at (1.5,-1) {$-\bnu_{\partial}$};
		\draw[->, color=red, ultra thick] (3.83, 3.5) -- ++(70:1.1);
		\node[above, color=red] at (3.55,4.3) {$\bnu_h$};
		\draw[->, color=red!70, ultra thick] (3.83, 3.5) -- ++(220:1.1);
		\node[above, color=red!70] at (3.77,2.15) {$\bnu_h$};
		\draw[->, color=red, ultra thick] (11, 9.4) -- ++(70:1.1);
		\node[above, color=red] at (11,10.35) {$\bnu_{\partial}$};
	\end{tikzpicture}
	\caption{The conormals  $\pm \bnu_{\partial}$ at the time slab boundaries and the interior conormals  $\bnu_{h}\restrict{K_S}$. Due to the non-smoothness of $S_h$, in general at a common face $F$ one has $(\bnu_{h}\restrict{K_S^1})\restrict{F}\neq -(\bnu_{h}\restrict{K_S^2})\restrict{F}$.}
	\label{conormal}
\end{figure}

Let  $F \in \mathcal{F}_I^n$ be an interior curved boundary face  with $F=K_S^1\cap K_S^2$, $K_S^1,K_S^2\in \mathcal{T}_{S_h^n}$. For a function $v$ defined on $K_S^1 \cup K_S^2$  we define the (vector valued) conormal jump on $F$:
\begin{equation} \label{def_conormal_jump}
	{\left[v\right]_{\bnu}}\restrict{F} := 
	\big(v\restrict{K_S^1}{\bnu_h}\restrict{K_{S}^1} + v\restrict{K_S^2}{\bnu_h}\restrict{K_{S}^2}\big)\restrict{F}.
\end{equation}
Note that in general for the Lipschitz surface $S_h^n$ we do not have $C^1$ smoothness across $F$ and thus ${\bnu_h}\restrict{K_{S}^1} \neq - {\bnu_h}\restrict{K_{S}^2}$ on $F$. If the function $v$ is continuous across $F$ we have ${[v]_{\bnu}}\restrict{F} = v\restrict{F} {[1]_{\bnu}}\restrict{F}$.   
For one-sided values at the boundary of a time slab we use the standard notation
\begin{align*}
	u_+^n& \coloneqq u_+(\cdot ,t_n)=\lim_{\eta\searrow 0}u(\cdot, t_n+\eta), \quad n =0,\dots,N-1,\\
	u_-^n& \coloneqq u_-(\cdot ,t_n)=\lim_{\eta\searrow 0}u(\cdot, t_n-\eta), \quad n=1,\dots,N
\end{align*}
and $u_-^0\coloneqq 0$.
Besides a possible discontinuity in the (finite element) functions between time slabs we also can have a  discontinuity in the space-time surface approximation $S_h$ between $S_h^n$ and $S_h^{n+1}$, cf. \Cref{remconsist}. In view of this we introduce, for $u$ defined on $ S_h^n \cup S_h^{n+1}$ the jump
\begin{equation}
	[u]_h^n\coloneqq u_+^n-u_-^n\circ \Theta_h^n \circ (\Theta_h^{n+1})^{-1} \quad \text{on}~\Gamma_h^{n+1}(t_n).\label{discrete_jumps}
\end{equation}
For $v\in H^1(S_h^n)$ we define the discrete material derivative as
\begin{equation}
	\mathring{v}\coloneqq \bw_S\cdot \nabla_{S_h}v=(\mathbf{P}_{S_h}\bw_S)\cdot \nabla_{S_h} v, \textcolor{black}{\qquad\mathbf{P}_{S_h}\coloneqq \bI - \bn_{S_h} \bn_{S_h}^T.}\label{weakmatderivative}
\end{equation}
This derivative is a.e. tangential to the  space-time surface approximation $S_h$. In the next theorem we give a discrete analogue of the partial integration formula \eqref{PIcont}. 
A proof of this result can be found in \cite[Appendix A]{sass2022}.
\begin{theorem}\label{PI_Theorem}
	On $S_h^n$,  $n=1,\dots,N$, we consider a strictly positive piecewise smooth function $\alpha_h\in\bigoplus_{K_S\in \Tau_{S_h^n}}C^1(K_S)$ and $R\coloneqq \alpha_h^{-1}\bw_S\cdot \bnu_{\partial }$. For $u,v\in H^1(S_h^n)$, the following identity holds:
\begin{align}
	\begin{aligned}
		&\int_{S_h^n}\frac{1}{\alpha_h}\mathring{u}v\dif \sigma_h=-\int_{S_h^n}\frac{1}{\alpha_h}u\mathring{v}\dif \sigma_h +\int_{\Gamma_h^n(t_n)}u_{-}^{n}v_{-}^{n}R_-^n\dif s_h\\
		&\qquad-\int_{\Gamma_h^n(t_{n-1})}u_{+}^{n-1}v_{+}^{n-1}R_+^{n-1}\dif s_h  + \sum_{F\in \mathcal{F}_I^n}\int_{F} u v \bw_S\cdot \left[\alpha_h^{-1}\right]_{\bnu} \dif F\\
		&\qquad  -\sum_{K_S\in  \mathcal{T}_{S_h^n}}\int_{K_S} uv\Div_{S_h}\left( \frac{1}{\alpha_h} \bP_{S_h}\bw_S\right)\dif \sigma_h.\label{PI-unschoen}
	\end{aligned}
\end{align}
\end{theorem}
We briefly comment on the result \eqref{PI-unschoen}. Comparing this identity with its continuous analogue \cref{PIcont}, we observe a similar structure. The general function $\alpha_h$ serves as a discrete variant of the term $\alpha$. Due to the non-smoothness of $S_h^n$, the identity \eqref{PI-unschoen} reveals several perturbations which are not present in \eqref{PIcont} but are natural. The identity \cref{PI-unschoen} is equal to \eqref{PIcont} if we neglect geometry errors and consider $\alpha_h=\alpha$, in particular we then have $S_h=S$, $R=1$ and $[\alpha_h^{-1}]_{\bnu}\restrict{F}=0$. Similar conormal jump terms between elements as in \cref{PI-unschoen} naturally occur in the DG framework as well.

\subsection{Transformation between space-time surface differential operators}

In this section we derive useful transformation formulas between continuous and discrete space-time differential operators, similar to the stationary setting, as in e.g., \cite[Section 2.3]{demlow2007adaptive}, where the identity
\begin{equation}
	(\nabla_\Gamma v_h^l)^e=\left(\bI-\delta \gh\right)^{-1}\tilde{\gp}_h\nabla_{\Gamma_h}v, \quad v_h\in H^1(S_h^n),\label{trafo_demlow_grad}
\end{equation}
with $\tilde{\gp}_h\coloneqq \bI-\frac{\n_h\n^T}{\n_h\cdot \n}$ is shown. 
 On $S_h$ we define
\begin{equation}
	\ga_h\coloneqq \frac{\mu_h}{\sqrt{1+V_h^2}} \tilde{\gp}_h(\bI -\delta \gh)^{-2}\tilde{\gp}_h.\label{ah_def}
\end{equation}
As in \cite[Section 2.3]{demlow2007adaptive}, we obtain from \cref{formelmuh}, \Cref{fancyformelcoro} and \cref{trafo_demlow_grad} the relation
\begin{equation}
	\int_{S^n}\frac{1}{\alpha }\nabla_{\Gamma}u_h^l\cdot\nabla_{\Gamma}v_h^l \dif \sigma =\int_{S_h^n}\ga_h\nabla_{\Gamma_h}u_h\cdot \nabla_{\Gamma_h}v_h\dif \sigma_h, \quad  u_h,v_h\in H^1(S_h^n).\label{Trafo_With_A_h}
\end{equation} 
We establish a similar relation as in \eqref{trafo_demlow_grad} but now with space-time gradients. The result \eqref{trafo_demlow_grad} can not directly be applied in our space-time setting, as the projection $\bp$ acts orthogonal to $\Gamma(t)$, but not to $S$. 
We introduce the following $(4 \times 4)$-matrices defined on $U$:
\begin{equation}
	\gp_0\coloneqq \mathbf{I}-\frac{\n_{S_h}\n_0^T}{\n_{S_h}\cdot \n_0},\quad \tilde{\gd}\coloneqq \begin{pmatrix}
		\mathbf{I}-\delta \mathbf{H} & -V^e\n\\
		-\delta \ddt{\n}{t}^T &(\alpha^e)^2
	\end{pmatrix},\label{D_tilde}
\end{equation}
where $\n_0^{T}=(\n^{T},0)$. Using  $\bn \cdot \ddt{\n}{t}=0 $ it follows that $\det (\tilde \gd)=(\alpha^e)^2 \det (\mathbf{I}-\delta \mathbf{H}) > 0$, hence $\tilde \gd$ invertible.
\begin{lemma}\label[lemma]{l6}
	For $1 \leq n \leq N$, the following holds:
	\begin{equation}
		(\nabla_Sv_h^l)^e=\tilde{\gd}^{-1}\gp_0\nabla_{S_h}v_h,  \quad v_h\in H^1(S_h^n).\label{2.2.19}
	\end{equation}
\end{lemma}

\begin{proof}
We have
    \[
	D_{(\bx,t)} \p= \begin{pmatrix}
		(\mathbf{I}-\mathbf{n}\mathbf{n}^T-\delta\mathbf{H}) & V^e\n-\delta \ddt{\n}{t} \\
		0& 1 
	\end{pmatrix} =:\gd_{\p}.
	\]
	For the spatially extended space-time normal vector $\n_{S}^e:= \frac{1}{\alpha^e} \begin{pmatrix} \n \\ - V^e \end{pmatrix}$
	we define
	\[
	\gp_S^e := \bI - \n_{S}^e (\n_{S}^e)^T=\begin{pmatrix} \bI -\frac{\n\n^T}{(\alpha^e)^2}& \frac{V^e\n}{(\alpha^e)^2} \\
		\frac{V^e\n^T}{(\alpha^e)^2}&\frac{1}{(\alpha^e)^2} \end{pmatrix}.
\]
We have the relation $\gd_{\p}^T=\tilde{\gd}\gp_S^e$. For  $v_h\in H^1(S_h^n)$,  the chain rule yields 
	\[
		\nabla_{(\x,t)}v_h^l=\nabla_{(\x,t)}(v_h^l\circ \p)=\gd_{\p}^T\big(\nabla_{(\x,t)}v_h^l\big)
		\circ \p.
	\]
	Hence,
	\begin{equation}
		(\nabla_Sv_h^l) \circ \p=\tilde{\gd}^{-1}\nabla_{(\x,t)}v_h^l.\label{2.2.16}
	\end{equation}
	Using $\gd_{\p} \n_0 =0$ on $S_h^n$ we get $\nabla_{(\x,t)}v_h^l= \gp_0\nabla_{(\x,t)}v_h^l$ and combining this with $\gp_0=\gp_0 \bP_{S_h}$  we get 
	\begin{equation*}
		(\nabla_Sv_h^l)\circ \p=\tilde{\gd}^{-1}\gp_0\nabla_{S_h}v_h \quad \text{on}~~ S_h^n.
	\end{equation*}
\end{proof}
\section{Basic geometric error estimates} \label{secGeometry}
Recall that the approximate space-time surface is based on the parametric mesh deformation mapping $\Theta_h^n$, cf.~\ref{Sapprox}. Properties of this mapping and the related ``ideal'' mapping $\Psi$, cf. Section~\ref{secSurface}, 
are derived in the recent work \cite{HeimannLehrenfeld}. Using these properties, bounds for the distance between $S$ and $S_h$ and between normals on $S$ and $S_h$ can be derived. Below in Lemma~\ref{l1} we collect a few main results from \cite{HeimannLehrenfeld} that are useful in  our analysis. Based on these results we derive further estimates for geometric quantities (Lemma~\ref{l_mu} and \ref{l3}) and for differences between derivatives of a function on $S_h$ and the derivatives of the corresponding lifted version on $S$. 

Clearly, for higher order approximation results one needs assumptions on the accuracy of $\phi_h \approx \phi$. The following assumptions are introduced in \cite[Assumptions 2.4, 2.5, 2.6]{HeimannLehrenfeld}.
\begin{assumption}\label[assumption]{ass_phi_h}
We assume an approximation ${\phi_h\in V_h^{k_{g,\spa},k_{g,\ti}}}$ of   $\phi$ that satisfies for all $1 \leq n \leq N$ and  $(m_s,m_q) \in \big( \{0,\ldots,k_{g,\spa}+1\} \times \{0,1\}\big) \cup \big( \{0\} \times \{0,\ldots, k_{g,\ti}+1\} \big)$ 
\begin{equation} \label{AssL1}
	\|D^{m_s} \partial_t^{m_q}(\phi_h-\phi)\|_{L^\infty(\T_n^\Gamma \times I_n)}\lesssim h^{k_{g,\spa}+1-m_s}+\Delta t^{k_{g,\ti}+1-m_q}.
\end{equation}
Represent $\phi_h$ as $\phi_h(\bx,t)= \sum_{m=0}^{k_{g,\ti}} \mathcal{X}_{\tau_m^n}(t) \phi_{h,\tau_m^n}(\bx)$, $t\in I_n$, $\phi_{h,\tau_m^n}(\bx):= \phi_h(\bx,\tau_m^n)$. We assume that for all $1 \leq n \leq N$ there exists  $\phi_{\Delta t}(\bx,t) = \sum_{m=0}^{k_{g,\ti}} \mathcal{X}_{\tau_m^n}(t) \phi_{\tau_m^n}(\bx)$, with $\phi_{\tau_m^n} \in C(\Omega) \cap C^{k_{g,s}+1}(\Omega_n^\Gamma)$  such that the following holds:
\begin{equation} \label{AssL2}
	\|D^{m_s}(\phi_{h,\tau_m^n}-  \phi_{\tau_m^n})\|_{L^\infty(\T_n^\Gamma)}\lesssim h^{k_{g,\spa}+1-m_s}, \quad 0 \leq m_s \leq k_{g,s}+1.
\end{equation}
\textcolor{black}{We assume that for all $1 \leq n \leq N$ there is $\phi_H \in C^{k_{g,\ti}+1}(I_n;V_h^{k_{g,s}})$ such that for $m_q = 0,\dots, k_{g,q} + 1$ and $m_s\in \{0,1\}$:
\begin{align}
\norm{\partial_t^{m_q}(\phi_H - \phi_h)}_{L^\infty(\Tau_n^\Gamma\times I_n)}&\lesssim \Delta t^{k_{g,q}+1-m_q}\label{AssDelta_t1}\\
\norm{D^{m_s}\partial_t^{m_q}(\phi_H - \phi_h)}_{L^\infty(\Tau_n^\Gamma\times I_n)}&\lesssim h^{k_{g,s}+1-m_s} + \Delta t^{k_{g,q}+1-m_q}.\label{AssDelta_t2}
\end{align}}
\end{assumption}
Note that these assumptions are satisfied if $\phi$ is sufficiently smooth and $\phi_h=I^tI^s \phi$ is a tensor product space time nodal interpolation in the finite element space $V_h^{k_{g,\spa},k_{g,\ti}}$. One can then take $\phi_{\Delta t}=I^t \phi$ and $\phi_H=I^s \phi$.
In the remainder of this paper we assume that Assumption~\ref{ass_phi_h} holds. The following results are derived in \cite{HeimannLehrenfeld}.
\begin{lemma}\label[lemma]{l1}
	\begin{align}
	\norm{\Theta_h^n-\Psi}_{L^\infty(\T_n^\Gamma \times I_n)}&\lesssim h^{k_{g,\spa}+1} + \Delta t^{k_{g,\ti}+1},\label{theta_psi_est}\\
	\norm{D\Theta_h^n-D\Psi}_{L^\infty(\T_n^\Gamma \times I_n)}&\lesssim h^{k_{g,\spa}} + \Delta t^{k_{g,\ti}+1},\label{Dtheta_psi_est}\\
		\norm{\Theta_h^n-\id}_{L^\infty(\T_n^\Gamma \times I_n)}&\lesssim h^{2} + \Delta t^{k_{g,\ti}+1},\label{theta_to_id}\\
      \norm{\delta}_{L^\infty(S_h)}&\lesssim h^{k_{g,\spa}+1} + \Delta t^{k_{g,\ti}+1},\label{dist_est} \\
		\norm{\n-\n_h}_{L^\infty(S_h)}&\lesssim h^{k_{g,\spa}} + \Delta t^{k_{g,\ti}},\label{spa_normal_est}\\
			\norm{\n_S^e-\n_{S_h}}_{L^\infty(S_h)}&\lesssim \textcolor{black}{h^{k_{g,\spa}} + \Delta t^{k_{g,\ti}}}. \label{st_normal_est}
\end{align}
	\end{lemma}
\begin{proof}
The results \eqref{theta_psi_est} and \eqref{Dtheta_psi_est} are given in \cite[Theorem 4.12]{HeimannLehrenfeld} and the result \eqref{theta_to_id} follows from  \cite[Corollary 4.2]{HeimannLehrenfeld} combined with  \eqref{theta_psi_est}. From \eqref{theta_psi_est} and \eqref{AA1} one easily derives the result \eqref{dist_est}, cf. \cite[Lemma 3.8]{lehrenfeld2017analysis}. The result \eqref{spa_normal_est}is obtained from \eqref{theta_psi_est} and \eqref{Dtheta_psi_est} with straightforward arguments, cf. \cite[Lemma 3.3]{grande2016higher}. Concerning \eqref{st_normal_est}  we note the following. In \cite[Lemma 5.12]{HeimannLehrenfeld} the estimate
\begin{equation} \label{estH}
 \norm{\n_S \circ \Phi_h^{st}-\n_{S_h}}_{L^\infty(S_h)} \lesssim \textcolor{black}{h^{k_{g,\spa}} + \Delta t^{k_{g,\ti}}}
\end{equation}
is proved. We refer to \cite{HeimannLehrenfeld} for the definition of  $\Phi_h^{st}$, which satisfies $\norm{\Phi_h^{st}-{\rm id}}_{L^\infty(S_h)} \lesssim h^{k_{g,\spa}+1} + \Delta t^{k_{g,\ti}+1}$ (cf. \cite[Corollary 5.9]{HeimannLehrenfeld}). Using this we obtain
\begin{align*}
 \norm{\n_S^e-\n_{S_h}}_{L^\infty(S_h)} & =\norm{\n_S \circ \bp -\n_{S_h}}_{L^\infty(S_h)}  \\ & \leq \|\n_S \circ \bp - \n_S \circ \Phi_h^{st}\|_{L^\infty(S_h)}+ \norm{\n_S \circ \Phi_h^{st}-\n_{S_h}}_{L^\infty(S_h)} \\
 & \lesssim \|\bp -{\rm id}\|_{L^\infty(S_h)} +\norm{\Phi_h^{st}-{\rm id}}_{L^\infty(S_h)} +  \norm{\n_S \circ \Phi_h^{st}-\n_{S_h}}_{L^\infty(S_h)}.
\end{align*}
Using $ \|\bp -{\rm id}\|_{L^\infty(S_h)}=\|\delta\|_{L^\infty(S_h)}$ and the estimate \eqref{estH} we obtain the result \eqref{st_normal_est}.
\end{proof}
\ \\[1ex]
\textcolor{black}{Without the assumptions \cref{AssDelta_t1} and \cref{AssDelta_t2} an extra term $h^{k_{g,s}+1}/\Delta t$ occurs in the bound in  \cref{st_normal_est}, cf. \cite[Remark 5.13]{HeimannLehrenfeld}.}

In the following lemma we derive bounds for differences between spatial surface measures on $S$ and $S_h$ and between space-time surface measures on $S$ and $S_h$, cf. Lemma~\ref{lemmeasure}.

\begin{lemma}\label[lemma]{l_mu} The following uniform estimates hold:
\begin{align} 
		\norm{1-\mu_h}_{L^\infty(S_{h})}&\lesssim  h^{k_{g,\spa}+1} + \Delta t^{k_{g,\ti}+1},\label{einminusmu_Space}\\
	\norm{1-\mu_h^S}_{L^\infty(S_{h})}&\lesssim \textcolor{black}{h^{k_{g,\spa}} + \Delta t^{k_{g,\ti}}}. \label{einminusmu_ST}
\end{align}
\end{lemma}
\begin{proof}
	Using the identity $\mathbf{n}\cdot \mathbf{n}_{h}=1-\frac{\norm{\mathbf{n}-\mathbf{n}_{h}}^2}{2}$ on $S_h$, the result \cref{formelmuh}, the uniform boundedness of the mean curvatures $\kappa_i$ and  \eqref{dist_est}-\eqref{spa_normal_est},  we obtain the estimate \eqref{einminusmu_Space}:
	\begin{align*}
		\norm{1-\mu_h}_{L^\infty(S_{h})}&=\norm{1-\left(1-\frac{\norm{\mathbf{n}-\mathbf{n}_{h}}^2}{2}\right)\prod_{i=1}^{2}(1-\kappa_i\delta)}_{L^\infty(S_{h})}\\
		&\lesssim \norm{\delta}_{L^\infty(S_{h})}+{\norm{\mathbf{n}-\mathbf{n}_{h}}^2_{L^\infty(S_{h})}} \lesssim  h^{k_{g,\spa}+1} + \Delta t^{k_{g,\ti}+1}.
	\end{align*}
	Using the identity \eqref{formelmuhn}, the result \eqref{ns_schreibweise} and the space-time normal bound \cref{st_normal_est} we obtain  
	\begin{align*}
		\norm{1-\mu_h^S}_{L^\infty(S_{h})}&=\norm{1-\alpha^e \, \n_0\cdot \mathbf{n}_{S_{h}}\prod_{i=1}^{2}(1-\delta\kappa_i)}_{L^\infty(S_{h})}\\
		&\lesssim  \norm{1-\alpha^e \, \n_0\cdot \mathbf{n}_{S}^e}_{L^\infty(S_{h})}+\textcolor{black}{h^{k_{g,\spa}} + \Delta t^{k_{g,\ti}}}\\
		& \lesssim \norm{1-\n_0\cdot\begin{pmatrix}\mathbf{n} \\-V^e \end{pmatrix}}_{L^\infty(S_{h})}+\textcolor{black}{h^{k_{g,\spa}} + \Delta t^{k_{g,\ti}}} \\
		&\lesssim \norm{1-\n\cdot \n}_{L^\infty(S_{h})}+\textcolor{black}{h^{k_{g,\spa}} + \Delta t^{k_{g,\ti}}}   \lesssim \textcolor{black}{h^{k_{g,\spa}} + \Delta t^{k_{g,\ti}}},
	\end{align*}
which proves  the estimate \eqref{einminusmu_ST}.
\end{proof}

\begin{remark}\label[remark]{delicate_alpha} \rm
Note that there is a loss of one power of $h$ and of $\Delta t$ in the bound for the space-time surface measure difference 
in \eqref{einminusmu_ST} compared to the bound for the space surface measure difference  in \eqref{einminusmu_Space}. These bounds 
are sharp and the reason for the loss of these powers is the following. 
	The function $\mu_h(\x,t)=\frac{\dif s(\p(\x,t)}{\dif s_h(\x,t)}$ is the quotient of the surface measures 
	of the interfaces $\Gamma(t)$ and $\Gamma_h^n(t)$, $n=1, \ldots,N$. The correspondence of points on $\Gamma_h(t)$ with 
	points on $\Gamma(t)$ is given by the \emph{orthogonal}  (spatially) closest point projection $\bp$. 
	The function $\mu_h^S(\x,t)=\frac{\dif \sigma(\p(\x,t)}{\dif \sigma_h(\x,t)}$ is the quotient of the surface 
	measures of the surfaces $S$ and $S_h^n$, $n=1,\ldots,N$. For the correspondence of points on $S_h^n$ with points on $S$ 
	the same projection $\bp$ is used, which, however, is in general \emph{not} an orthogonal projection on $S$. This loss of the 
	orthogonality property causes the loss of one power of $h$ and of $\Delta t$.
	\end{remark} 
	\begin{remark} \label[remark]{remVh} \rm
	Using \eqref{formelmuh}, \eqref{formelmuhn} and \eqref{nsh_repres} we get 
	\[
		\frac{\mu_h^S}{\mu_h}=\sqrt{1+ (V^2)^e} \, \frac{\bn_0 \cdot \bn_{S_h}}{\bn \cdot \bn_h} =\frac{\sqrt{1+(V^e)^2}
		}{\sqrt{1+V_h^2}}\quad \text{on }S_h^n.
	\]
	Using this and \eqref{einminusmu_Space}-\eqref{einminusmu_ST} we see that (for $ h \lesssim \Delta t$)  $V_h$ is an approximation of $V^e$ of 
	order (only) $\mathcal{O}(h^{k_{g,\spa}}+\Delta t^{k_{g,\ti}})$. In the application in \Cref{section_application} we need a more accurate approximation of $V^e$.
\end{remark}
\ \\[1ex]
A further  geometric quantity relevant for the accuracy of the space-time surface approximation $S_h$ is the  conormal jump ${[\cdot]_{\bnu}}_{|F}$ across interior faces $ F \in \mathcal{F}_I$. In the next lemma we derive bounds for this quantity. 
\begin{lemma}\label[lemma]{l3}
	For any curved face $F\in \mathcal{F}_I$ we have
	\begin{align}
		\norm{{[1]_{\bnu}}_{|F}}_{L^\infty(F)}&\lesssim \textcolor{black}{h^{k_{g,\spa}}+\Delta t^{k_{g,\ti}}}, \label{InnerJumpnE}\\
		\norm{\gp_S^e{[1]_{\bnu}}_{|F}}_{L^\infty(F)}&\lesssim \textcolor{black}{h^{2k_{g,\spa}} + \Delta t^{2k_{g,\ti}}}. \label{InnerJumpnEImproved}
	\end{align}
	The conormal jump $[1]_{\bnu}$ is defined in \cref{def_conormal_jump}.
\end{lemma}
\begin{proof} 
Take an arbitrary $F\in \mathcal{F}_I$ and a fixed $(\x,t)\in F$. Let $\s_1,\s_2$ be two orthonormal vectors that span 
the tangent plane $\hat{F}$ of the face $F$ at $(\x,t)$.
We omit the argument $(\x,t)$ in this proof. Let $K_S^1$, $K_S^2\in \mathcal{T}_{S_h}$ be the elements of 
the triangulation of the space-time surface with the common face $F$. We define the space-time normals 
on $K_S^1$, $K_S^2$ by $\n_1:=\n_{S_h^1}$ and $\n_2:=\n_{S_h^2}$ respectively. The corresponding conormals for $F$ are 
denoted 
by $\bnu_1$ and $\bnu_2$. Note that 
	$\left\lbrace \n_1, \bnu_1, \s_1, \s_2\right\rbrace$ and $\left\lbrace\n_2, \bnu_2, \s_1, \s_2\right\rbrace$
form orthonormal bases in $\R^4$ and $\hat F^\perp = {\rm span}\left\lbrace\n_1, \bnu_1\right\rbrace = 
{\rm span}\left\lbrace\n_2, \bnu_2\right\rbrace$.  We define the orthogonal matrix 
\begin{equation*}
	Q\coloneqq \begin{pmatrix} \n_1& \bnu_1& \s_1 &\s_2 \end{pmatrix}
	\begin{pmatrix} \bnu_1& -\n_1& \s_1 &\s_2 \end{pmatrix}^T= \n_1\bnu_1^T-\bnu_1(\n_1)^T+\s_1\s_1^T +\s_2 \s_2^T,
\end{equation*}
 which is a rotation by $\frac{\pi}{2}$ in the two-dimensional plane $\hat F^\perp$. 
 We have  
	$Q\bnu_1=   \n_1
$ and $\bnu_2 \cdot Q \bnu_2=0$.
From the latter and $F^\perp={\rm span}\left\lbrace\n_2, \bnu_2\right\rbrace$ it follows that $Q \bnu_2 = \pm \bn_2$. 
By construction the angle between $\bn_1$ and $\bn_2$ is the same as between  $\bnu_1$ and $- \bnu_2$,
cf. \Cref{cosine_angles}. Hence, $\bn_1 \cdot \bn_2= - \bnu_1 \cdot \bnu_2= -(Q\bnu_1)\cdot (Q\bnu_2)=- 
\bn_1 \cdot Q \bnu_2$. We conclude that $Q \bnu_2=-\bn_2$ holds.

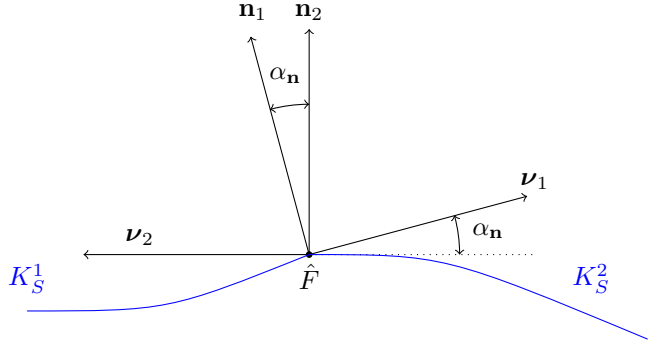
\begin{figure}[!htbp]
\centering
\begin{tikzpicture}[scale = 1.5]
	\newcommand{\ptsize}{0.7pt}
\filldraw circle(\ptsize) (0,0);
\draw[->] (0,0)-- (-2,0);
\draw[dotted] (0,0)-- (2,0);
\draw[->] (0,0) -- ++(15:2);
\draw[->] (0,0) -- ++(90:2);
\draw[->] (0,0) -- ++(105:2);
\draw[blue] (0,0)  .. controls (1.2,0) .. (3,-0.75);
\draw[blue] (0,0)  .. controls (-1.25,-0.5) .. (-2.5,-0.5);

\coordinate (a) at (1,0);
\coordinate (b) at (0,0);
\coordinate (c) at (1,0.27);
\pic["$\alpha_{\mathbf{n}}$",draw=black, <->, angle eccentricity=1.2, angle radius=2cm]
    {angle=a--b--c};
\coordinate (aa) at (0,1);
\coordinate (bb) at (0,0);
\coordinate (cc) at (-0.27,1);
\pic["$\alpha_{\mathbf{n}}$",draw=black, <->, angle eccentricity=1.2, angle radius=2cm]
    {angle=aa--bb--cc};

\node[scale=1, below] at (0, 0) {$\hat{F}$};
\node[scale=1, below, color=blue] at (2.5, 0) {$K_S^2$};
\node[scale=1, below, color=blue] at (-2.5, 0) {$K_S^1$};
\node[scale=1, above] at (-1.5, 0) {$\bnu_2$};
\node[scale=1, above] at (-0.5, 2) {$\n_1$};
\node[scale=1, above] at (0, 2) {$\n_2$};
\node[scale=1, above] at (2, 0.5) {$\bnu_1$};
\end{tikzpicture}
\caption{Illustration of normals and conormals for two neighbouring elements $K_S^1, K_S^2\in \mathcal{T}_{S_h}$.}
\label{cosine_angles}
\end{figure}

Now let $\bn_S^e = \bn_S^e(\bx,t)$ be the extended space-time normal to $S$, at $(\bx,t) \in F$. Using \eqref{st_normal_est}
 we get
\begin{equation} \label{A2} \begin{split}
		\norm{[1]_{\bnu}(\bx,t)}&=\norm{\bnu_1+\bnu_2}=\norm{Q\left(\bnu_1+\bnu_2\right)}=\norm{\n_1-\n_2} \\
		&\lesssim \norm{\n_1-\n_S^e}+\norm{\n_S^e-\n_2}\lesssim \textcolor{black}{h^{k_{g,\spa}} + \Delta t^{k_{g,\ti}}}, 
\end{split}
\end{equation}
which yields \cref{InnerJumpnE}. We turn to the proof of \cref{InnerJumpnEImproved}. 
We write $Q\gp_S^e[1]_{\bnu}= \alpha_1 \bn_1 +\alpha_2 \bnu_1 + \beta_1 \bs_1 + \beta_2 \bs_2$. A straightforward computation 
yields 
\begin{align*}
	 \alpha_1 & = 1+\bnu_1 \cdot \bnu_2 - \left(\bnu_1 \cdot \n_S^e\right)^2-\bnu_1 \cdot \n_S^e \, \n_S^e \cdot \bnu_2 \\ 
	 \alpha_2 & =-\n_1 \cdot \bnu_2+\n_1 \cdot \n_S^e \, \n_S^e \cdot \left(\bnu_1+\bnu_2\right)\\
	 \beta_i & = \s_i \cdot  \n_S^e \, \n_S^e\cdot (\bnu_1+\bnu_2), \quad i=1,2.
	\end{align*}
	Hence, $\norm{\gp_S^e[1]_{\bnu}}=\norm{Q\gp_S^e[1]_{\bnu}}= \left( \alpha_1^2 +\alpha_2^2 +\beta_1^2+\beta_2^2\right)^\frac12$.
	We introduce the notation $\epsilon:=(1+ \frac{h}{\Delta t} )h^{k_{g,\spa}} + \Delta t^{k_{g,\ti}}$ and use (cf. \eqref{A2})
	\begin{align*}
	  \norm{\bnu_1 +\bnu_2}  & \lesssim \epsilon, \quad |1+\bnu_1 \cdot \bnu_2| = \tfrac12\norm{\bnu_1 +\bnu_2}^2 
	  \lesssim \epsilon^2,\\
	  |\bnu_i \cdot \bn_S^e| & = |\bnu_i \cdot (\bn_S^e - \bn_i)| \lesssim \epsilon, \quad i=1,2,\\
	   |\bs_i \cdot \bn_S^e| & = |\bs_i \cdot (\bn_S^e - \bn_i)| \lesssim \epsilon, \quad\, i=1,2.
	\end{align*}
Thus we get $|\alpha_1| \lesssim \epsilon^2$ and $|\beta_i| \lesssim \epsilon^2$, $i=1,2$. It remains to bound $\alpha_2$. Due to $\n_1\cdot \n_S^e=1-\frac12\norm{\n_1-\n_S^e}^2= 1+\mathcal{O}(\epsilon^2)$, cf. \cref{st_normal_est}, and $\n_1\cdot \bnu_1=0$ we obtain
\begin{equation*}
	|\alpha_2|\lesssim |(\n_S^e-\n_1)\cdot(\bnu_1+\bnu_2)|\lesssim \epsilon^2.
\end{equation*}
Combining the results yields the estimate \cref{InnerJumpnEImproved}.
\end{proof}

We now derive an estimate for the difference between a tangential gradient on $S_h$ and the corresponding gradient on $S$ of the lifted function. This useful result is an easy consequence of the relation given in Lemma~\ref{l6}. 
\begin{lemma} \label{gradestimate}
For all $v_h \in H^1(S_h^n)$ the following uniform estimate holds:
\begin{equation}
 \norm{\nabla_{S_h}v_h- \big( \nabla_S v_h^l\big)^e}_{L^2(S_h^n)} \lesssim \big(\textcolor{black}{h^{k_{g,\spa}} + \Delta t^{k_{g,\ti}}}\big) \norm{\nabla_{S_h} v_h}_{L^2(S_h^n)}.
\end{equation}
\end{lemma}
\begin{proof}
From \eqref{2.2.19} we obtain, with $\bP_0$ and $\tilde\bD$ as in \eqref{D_tilde},
\begin{equation} \label{AD}
 \nabla_{S_h} v_h - (\nabla_S v_h^l\big)^e = \nabla_{S_h} v_h - \tilde \bD^{-1} \bP_0  \nabla_{S_h} v_h = (\bP_{S_h}- \tilde \bD^{-1} \bP_0)  \nabla_{S_h} v_h. 
\end{equation}
From \eqref{2.2.19} it follows that $\tilde \bD^{-1} \bP_0= \bP_S^e \tilde \bD^{-1} \bP_0$ holds.  The definition of $\tilde{\gd}$ yields
\[
	\tilde{\gd}=\bI-\alpha^eV^e\begin{pmatrix}
		0 & \n_S^e \end{pmatrix}+ \tilde{\gd}_\Delta, \quad \tilde{\gd}_\Delta\coloneqq  \begin{pmatrix} -\delta \gh & 0 \\ -\delta \ddt{\n}{t} & 0\end{pmatrix},\label{Ad_schlange_delta_darstellung}
\]
and thus $\bP_S^e \tilde \bD= \bP_S^e+ \bP_S^e\tilde{\gd}_\Delta$, i.e., $\bP_S^e \tilde \bD^{-1}= \bP_S^e- \bP_S^e\tilde{\gd}_\Delta \tilde \bD^{-1}$. 
Hence, 
\begin{equation} \label{loc77} \begin{split}
 \bP_{S_h}- \tilde \bD^{-1} \bP_0  & = \bP_{S_h}- \bP_{S}^e\tilde \bD^{-1} \bP_0=\bP_{S_h}-\bP_{S}^e  \bP_0 + \bP_{S}^e\tilde{\gd}_\Delta\tilde \bD^{-1} \bP_0 \\
 & = \bP_{S_h}-\bP_{S}^e + \bP_{S}^e (\bI- \bP_0) + \bP_{S}^e\tilde{\gd}_\Delta\tilde \bD^{-1} \bP_0.
\end{split}
\end{equation}
From \eqref{st_normal_est} we get $\|\bP_{S_h}-\bP_{S}^e\|_{L^\infty(S_h^n)} \lesssim \textcolor{black}{h^{k_{g,\spa}} + \Delta t^{k_{g,\ti}}}$ and $\|\bP_{S}^e (\bI- \bP_0)\|_{L^\infty(S_h^n)}  \\ = \norm{\frac{\bP_S \bn_{S_h}\bn_0^T} {\bn_{S_{h}}^T\bn_0}}_{L^\infty(S_h^n)} \lesssim \textcolor{black}{h^{k_{g,\spa}} + \Delta t^{k_{g,\ti}}}$. With $\det (\tilde \gd)=(\alpha^e)^2 \det (\mathbf{I}-\delta \mathbf{H})$ and the definition of $\tilde \bD$ it follows that (for $h$ and $\Delta t$ sufficiently small) $\|\tilde \bD^{-1}\|_{L^\infty(S_h^n)} \lesssim 1$. With \eqref{dist_est} we get $\|\tilde{\gd}_\Delta\|_{L^\infty(S_h^n)}\lesssim h^{k_{g,\spa}+1} + \Delta t^{k_{g,\ti}+1}$. Hence, 
$\norm{ \bP_{S}^e\tilde{\gd}_\Delta\tilde \bD^{-1} \bP_0}_{L^\infty(S_h^n)}$ $\lesssim h^{k_{g,\spa}+1} + \Delta t^{k_{g,\ti}+1}$. Using  these estimates  in \eqref{loc77} we obtain
\begin{equation} \label{ADD}
 \norm{\bP_{S_h}- \tilde \bD^{-1} \bP_0}_{L^\infty(S_h^n)}\lesssim \textcolor{black}{ h^{k_{g,\spa}} + \Delta t^{k_{g,\ti}} }.
\end{equation}
Combining this with the result in \eqref{AD} completes the proof.
\end{proof}
\ \\[1ex]
The following corollary will be used in the consistency error analysis in Section~\ref{secconsistency}. Let $\bv_h: \, S_h \to \R^4$  be a vector-valued function with components given by $\bv_h=(v_{h,i}), {1\leq i \leq 4}$, and its surface gradient $\nabla_{S_h} \bv_h = \begin{pmatrix} \nabla_{S_h} v_{h,1}, \ldots, \nabla_{S_h} v_{h,4} \end{pmatrix}^T$.
\begin{corollary} \label{CorH}
 For all $\bv_h \in \big(W^{1,\infty}(S_h^n))^4$ the following uniform estimate holds:
\begin{equation} \label{estdiv}
 \norm{\div_{S_h}\bv_h- \big(\div_{S}\bv_h^l\big)^e}_{L^\infty(S_h^n)} \lesssim \big(\textcolor{black}{h^{k_{g,\spa}} + \Delta t^{k_{g,\ti}}}\big) \norm{\nabla_{S_h} \bv_h}_{L^\infty(S_h^n)}.
\end{equation}
\end{corollary}
\begin{proof}
 Note that $\div_{S_h}\bv_h= \tr \big(\bP_{S_h} \nabla_{(x,t)} \bv_h \bP_{S_h} \big)= \tr (\nabla_{S_h} \bv_h)$. Component-wise application of \eqref{AD} yields
 \[
   \div_{S_h}\bv_h- \big(\div_{S}\bv_h^l\big)^e= \tr \big(\nabla_{S_h} \bv_h  (\bP_{S_h}- \tilde \bD^{-1} \bP_0)^T\big).  
 \]
The result \eqref{estdiv} follows using \eqref{ADD}. 
\end{proof}

\subsubsection*{Choice of $\alpha_h$} 
In \Cref{PI_Theorem} we used a general function $\alpha_h$. For the formula \eqref{PI-unschoen} to be a discrete analogon of  \eqref{PIcont} we choose an  $\alpha_h$ that approximates the function $\alpha=(1+V^2)^{\frac12}$. A natural choice for this approximation seems $\alpha_h=(1+V_h^2)^{\frac12}$, which  has accuracy $\mathcal{O}(h^{k_{g,\spa}} + \Delta t^{k_{g,\ti}})$, cf. Remark~\ref{remVh}. In the application that we consider below (Section~\ref{section_application}) this approximation  turns out to be \emph{not} sufficiently accurate for an optimal order discretization error.
We now introduce a (natural) specific choice for $\alpha_h$ which has an  order of accuracy that is one higher than for $(1+V_h^2)^{\frac12}$.  Let $\tilde{\phi}_{h} \in V_h^{k_{g,\spa}+1,k_{g,\ti}+1}$ be a one order higher order approximation of the level set function $\phi$ than the one introduced in  \Cref{ass_phi_h}. More precisely, we assume an approximation ${\tilde \phi_h\in V_h^{k_{g,\spa}+1,k_{g,\ti}+1}}$ of   $\phi$ that satisfies for all $1 \leq n \leq N$ and  $(m_s,m_q) \in \big( \{0,\ldots,k_{g,\spa}+1\} \times \{0,1\}\big) \cup \big( \{0\} \times \{0,\ldots, k_{g,\ti}+1\} \big)$ 
\begin{equation} \label{AssL4}
	\|D^{m_s} \partial_t^{m_q}(\tilde \phi_h-\phi)\|_{L^\infty(\T_n^\Gamma \times I_n)}\lesssim h^{k_{g,\spa}+2-m_s}+\Delta t^{k_{g,\ti}+2-m_q}.
\end{equation}
We define
\begin{equation}
	 \alpha_h:=\sqrt{1 + \tilde V_h^2}, \quad \text{with}~\tilde{V}_h\coloneqq - \frac{\partial \tilde{\phi}_h}{\partial t} \norm{\nabla \tilde{\phi}_h}^{-1}.\label{vh_tilde_def}
\end{equation}
Note that the function $\tilde{V}_h$ does not use the mesh deformation and is   an approximation of order $\mathcal{O}(h^{k_{g,\spa}+1}+\Delta t^{k_{g,\ti}+1})$ of the normal velocity $V=\bw \cdot \bn$. \\
\emph{This choice for $\alpha_h$ is used  in the remainder of this work}. \\ A very similar one is used in the numerical experiments in \cite{sass2022}. The accuracy estimates \eqref{AssL4}  imply
\begin{alignat}{2}
   \norm{\alpha^e-\alpha_h}_{L^\infty(S_h)}&\lesssim h^{k_{g,\spa}+1} + \Delta t^{k_{g,\ti}+1}&&, \label{alpha_est}\\
	\norm{\sqrt{1+V_h^2}-\alpha_h}_{L^\infty(S_h)}&\lesssim h^{k_{g,\spa}} + \Delta t^{k_{g,\ti}}&&,\label{alpha_h_sqrtvh}\\
	\norm{\nabla_{(\bx,t)}(\alpha^e-\alpha_h)}_{L^\infty(K_S)}&\lesssim h^{k_{g,\spa}} + \Delta t^{k_{g,\ti}}&&, K_S\in \Tau_{S_h},\label{alpha_grad_est} \\
	\norm{\bP_S^e\left[\alpha_h^{-1}\right]_{\bnu |F}}_{L^\infty(F)} &\lesssim h^{k_{g,\spa}+1} + \Delta t^{k_{g,\ti}+1}&&, F\in \mathcal{F}.\label{alpha_h_good_jump}
\end{alignat}
{\color{black} The result \cref{alpha_h_good_jump} follows from \eqref{alpha_est}, $\left[\alpha^e\right]_{\bnu |F}=0$ and the estimate \eqref{InnerJumpnEImproved}. 
}
We finally derive  an estimate for  $R=\frac{1}{\alpha_h}\w_S\cdot \n_{\partial }$, with $\n_\partial=\frac{1}{\sqrt{1+V_h^2}}(V_h\n_h^T,1)^T$, which  appears in the partial integration formula in \Cref{PI_Theorem}.
\begin{lemma}\label[lemma]{l4}
	The following estimate holds:
	\begin{align}
		\norm{R-1}_{L^\infty(S_h)}&\lesssim \textcolor{black}{h^{k_{g,\spa}} + \Delta t^{k_{g,\ti}}}.\label{second_geom_est}
	\end{align}
\end{lemma}
\begin{proof}
	Note  the estimate $	\norm{\w^e\cdot \n -V_h}_{L^\infty(S_h)}\lesssim \textcolor{black}{h^{k_{g,\spa}} + \Delta t^{k_{g,\ti}}}$, which follows from \cref{st_normal_est} and \Cref{lemma_nsh_repres}. 	From \eqref{dist_est} and the smoothness of $\bw$ we get $\|\w_S - \w_S^e\|_{L^\infty(S_h)}\lesssim h^{k_{g,\spa}+1} + \Delta t^{k_{g,\ti}+1} $. Using this,  \cref{alpha_h_sqrtvh} and \cref{spa_normal_est} yields
\begin{align*}
	&\norm{R-1}_{L^\infty(S_h)}=\norm{\frac{\w_S\cdot \n_{\partial }}{\alpha_h}-1}_{L^\infty(S_h)}\lesssim \norm{\bw_S^e \cdot \n_{\partial } - \alpha_h}_{L^\infty(S_h)}+h^{k_{g,\spa}+1} + \Delta t^{k_{g,\ti}+1}\\
	&\quad \lesssim \norm{V_h\w^e\cdot \n_h+1-(1+V_h^2)}_{L^\infty(S_h)}+h^{k_{g,\spa}} + \Delta t^{k_{g,\ti}}\\
	&\quad \lesssim  \norm{\w^e\cdot \n -V_h}_{L^\infty(S_h)} +h^{k_{g,\spa}} + \Delta t^{k_{g,\ti}} \lesssim \textcolor{black}{h^{k_{g,\spa}} + \Delta t^{k_{g,\ti}}}.
\end{align*}
\end{proof}

\section{Application to a surface diffusive transport equation}\label{section_application}
We consider a well-known basic model for convection and molecular diffusion of a surface species, cf. \cite{James04,DEreview,GReusken2011}. The conservation of mass principle combined with Fick's law for the diffusive flux leads to the parabolic surface partial differential equation
\begin{equation} \label{surfactant1}  \begin{split}
		\dot{u}+ u\Div_{\Gamma}\bw-\mu_d \Delta_{\Gamma} u&=f\quad  \text{on } \Gamma(t),\quad t\in (0,T],
		\\
		u(\cdot,0)&=0\quad \text{on } \Gamma(0), 
\end{split} \end{equation}
for the scalar unknown function $u=u(\bx,t)$. Here $\mu_d >0$ denotes the constant diffusion coefficient and $\dot{u}=\frac{\pa u^e}{\pa t}$ + $\bw\cdot \nabla u^e$ the material derivative along the velocity field $\bw$,  cf. \eqref{def_weakdev_h1}.   
To simplify the error analysis below, we assume that the right-hand side $f$ is such that  $\int_{\Gamma(t)}f\dif s=0$, $t\in [0,T]$, holds. This condition is satisfied if \eqref{surfactant1} is the transformed variant of an equation as in \eqref{surfactant1} with a nonzero initial condition and a source term that is zero. From  the assumption on $f$ is follows that  a solution $u$ of \eqref{surfactant1} has the property $ \int_{\Gamma(t)} u \dif s=0$ for all  $t\in [0,T]$, which mimics a mass conservation property.

We introduce a variational space-time formulation of \cref{surfactant1}. A well-posed weak variational formulation  in a suitable broken Hilbert space $W_b$ is studied in \cite{olshanskii2014eulerian}. We do not specify this space here. Multiplying \cref{surfactant1} by test functions and integrating over the space-time slabs $S_n$ results in the following variational problem, with the usual notation $
[v]^n \coloneqq v_+^n-v_-^n\in L^2(\Gamma(t_n))$ and $[v]^0=v_+^0$. 
For a given $f\in L^2(S)$ determine $u\in W^b$
such that
\begin{align}
\begin{aligned} 
\sum_{n=1}^N B_n(u,v) &\coloneqq \sum_{n=1}^N\Big[\int_{S_n} \frac{1}{\alpha}\dot{u} v \dif \sigma+ \int_{S^n}\frac{1}{\alpha}\big( uv\Div_{\Gamma}\mathbf{w}
+\mu_d\nabla_{\Gamma}u\cdot \nabla_{\Gamma}v\big)\dif \sigma  \\ & \quad +\big([u]^{n-1},v_+^{n-1}\big)_{\Gamma(t_{n-1})}\Big]=\int_S\frac{1}{\alpha}fv\dif s\dif t \quad \text{for all}~v\in W^b.\label{prob_cont} 
\end{aligned}
\end{align}
Instead of space-time surface integrals of the form $\int_{S^n} \frac{1}{\alpha} g \dif \sigma$ one can also use the double integral $\int_{t_{n-1}}^{t_n} \int_{\Gamma(t)} g \dif s \dif t$. We use the surface variant with $\int_{S^n} \frac{1}{\alpha} g \dif \sigma$ because in the error analysis of the discretization method, that we treat below, we apply partial integration on the space-time surface $S_h^n$, cf. \Cref{PI_Theorem}. 
\begin{remark} \label{remint} \rm
Using partial integration 
\begin{align}
	\begin{aligned} 
	 & \int_{S^n}\frac{1}{\alpha}\dot{u}v\dif \sigma \\ 
	 & =-\int_{S^n}\frac{1}{\alpha}\left(u\dot{v}+uv\Div_{\Gamma}\bw\right)\dif \sigma +\left(u_-^{n},v_-^{n}\right)_{\Gamma(t_{n})}-\left(u_+^{n-1},v_+^{n-1}\right)_{\Gamma(t_{n-1})},\label{pi_cont}
	 \end{aligned}
 \end{align}
for $u,v\in C^1(S^n)$, the formulation in \eqref{prob_cont} can be written in different equivalent forms. The discretization method below will be based on the antisymmetric formulation in which the material derivative appears in the form $\int_{S^n}\frac{1}{2 \alpha}(\dot{u}v - u \dot{v})\dif \sigma$. 
\end{remark}

\subsection{Space-time trace finite element discretization} \label{secmethod}
In this section we introduce a fully discrete higher order space-time discretization of \cref{prob_cont}, similar to \cite{sass2022}. The method is based on the standard space-time tensor product finite element space $V_h^{k_\spa,k_\ti}$, cf. \eqref{st_fe_space}. To obtain a higher order feasible geometry approximation we combine this space with the  parametric mapping \eqref{Theta_Def} as follows. We consider the prisms that are cut by $S_{\lin}$, i.e. prisms in $Q_{h}^S$ (cf. Fig.~\ref{theta_picture}), and apply the mesh deformation to obtain the domains:
\begin{equation*}
Q_{\Theta,n}^{S}\coloneqq \Theta_h^n(Q^S_{n}),~n=1,\dots,N, \quad Q_{\Theta}^{S}\coloneqq\bigcup_{n=1}^N Q_{\Theta,n}^{S}.
\end{equation*}
We define the parametric finite element space
\begin{equation}
V_{h,\Theta}^{k_\spa,k_\ti}\coloneqq \left\lbrace v_h:Q_{\Theta}^{S}\rightarrow \R| (\bx,t)\mapsto v_h(\Theta_h^n(\bx,t))\in V_h^{k_\spa,k_\ti}\restrict{Q^S_{n}}, n=1,\dots,N \right\rbrace. \label{Vh_Theta_Def}
\end{equation}
We use degree $k_\spa$ and $k_\ti$ in space and time, respectively. These degrees may be different from the degrees $k_{g,s}$ and $k_{g,q}$ that are used in the parametric mapping $\Theta_h^n$, cf. \eqref{Theta_Def}. 
 We introduce the following bilinear form on the time slab $S_h^n$:
\begin{align}
\begin{aligned}\label{discreteformendetail}
B_{h,n}(u,v)\coloneqq &\int_{S_h^n}\frac{1}{\alpha_h}\left(\frac12\mathring{u}v-\frac12 u\mathring{v} +\frac12uv\Div_{\Gamma_h}\mathbf{w}+ \mu_d \nabla_{\Gamma_h}u\cdot \nabla_{\Gamma_h}v\right)\dif \sigma_h\\
&\qquad+\frac12\left(R_-^n u_-^n,v_-^n\right)_{\Gamma_h^n(t_n)}+\frac12\left(R_+^{n-1} u_+^{n-1},v_+^{n-1}\right)_{\Gamma_h^n(t_{n-1})}\\
&\qquad-\left( R_+^{n-1} u^{n-1}_-\circ \Theta_h^{n-1}\circ(\Theta_h^{n})^{-1},v_+^{n-1}\right)_{\Gamma_h^n(t_{n-1})}, 
\end{aligned}
\end{align} 
and its sum
\begin{equation}
B_h(u,v)\coloneqq \sum_{n=1}^N B_{h,n}(u,v).\label{bh_beta}
\end{equation}
In addition we  use two stabilization terms. The first one is a variant of the so-called volume normal derivative stabilization \cite{grande2018analysis,burmanembedded}: 
\begin{equation}
s_1(u,v)\coloneqq \xi_1 \int_{Q_{\Theta}^{S}}(\bn_h\cdot \nabla u) (\bn_h\cdot \nabla v)\dif\, (\x, t).\label{NormalStabi}
\end{equation}
With the average $\overline{u}_h^n(t):= \int_{\Gamma_h^n(t)} u \dif s_h  $ we define the second stabilization term
\begin{equation}
  s_2(u,v):= \xi_2 \sum_{n=1}^N \int_{t_{n-1}}^{t_{n}}\overline{u}_h^n \overline{v}_h^n \dif t.\label{PoincStabi}
\end{equation}
Based on the literature we take the parameter range 
\begin{equation}
h\lesssim\xi_1 \lesssim h^{-1}. \label{stabi_parameter_restrictions}
\end{equation}
The parameter range for $\xi_2$ is specified  in \eqref{stabiparameterbedingung} below. We add these stabilizations to the bilinear form:
\begin{equation}
B_h^{\rm stab}(u,v)\coloneqq B_h(u,v)+s_1(u,v)+ s_2(u,v). \label{stabbilform}
\end{equation}
Before defining the fully discrete problem we need a suitable approximation of the right hand side $f$. We assume $f_h\in L^2(S_h)$ to be an approximation of the exact data $f$ satisfying
\begin{equation}
\norm{f_h-\mu_h f^e}_{L^2(S_h)}\lesssim (h^{k_{g,\spa}+1} + \Delta t^{k_{g,\ti}+1})\,\textcolor{black}{\norm{f}_{L^2(S)}}.\label{f_approx}
\end{equation}
From the literature it is known how such a data approximation can be determined, see e.g. \cite[Remark 4.43]{Diss_Sass} or \cite[Remark 5]{grande2018analysis} in similar settings.

We now introduce the \emph{space-time discrete variational problem}:
Given a right-hand side $f_h\in L^2(S_h)$ that satisfies \cref{f_approx}, determine $u_h\in V_{h,\Theta}^{k_\spa,k_\ti}$ such that
\begin{equation}
B_h^{\rm stab}(u_h,v_h)=\int_{S_h} \frac{f_h v_h}{\sqrt{1+V_h^2}}\dif \sigma_h 
\quad \text{for all } v_h\in V_{h,\Theta}^{k_\spa,k_\ti}.\label{discreteproblem}
\end{equation}
In the discrete problem above we used the antisymmetric form of the material derivative, cf. Remark~\ref{remint}. An advantage of this form is that it almost immediately leads to a stability estimate in a reasonable norm, cf. Section~\ref{secStrang}.
An extensive discussion of implementation aspects of this method is given  in \cite[Subsection 3.7]{sass2022}.
Here we only briefly address a few points.  

The use of the function $R$ in the boundary terms in 
\cref{discreteformendetail}  is motivated by the partial integration  formula \cref{PI-unschoen} on $B_h^n$. 
These weights can be included in an implementation without significant additional computational costs. In the consistency error analysis below this weighting with $R$ plays a key role. We are not able  to derive optimal bounds if we replace $R$ by 1 and use the (sharp) estimates in Lemma~\ref{l4}.  On the other hand, so far in numerical experiments we observe optimal order convergence if we replace $R$ by 1. 
	
	The factor $\Theta_h^{n-1}\circ (\Theta_h^n)^{-1}$ that appears in the last term in \eqref{discreteformendetail} is included to (weakly) transfer the solution from  the previous time interval $I_{n-1}$ to the current one $I_n$. We need such a transfer operator because the approximate surface may be \emph{dis}continuous between time slabs, $\Gamma_h^{n}(t_{n-1}) \neq \Gamma_h^{n-1}(t_{n-1})$, cf. Remark~\ref{remconsist}.

	The volume normal derivative stabilization $s_1(\cdot,\cdot)$ is a standard technique in trace or cut finite element methods that is used to control the effects of small cuts. The other stabilization term $s_2(\cdot,\cdot)$ is motivated by a Poincar\'e type inequality. For this we recall that for $t \in [0,T]$ there are constants $c_{P,1}(t) >0$, $c_{P,2}(t)$ such that for all $u \in H^1(\Gamma(t))$ the    inequality $\|\nabla_\Gamma u \|_{L^2(\Gamma(t))}^2 \geq c_{P,1}(t) \|u\|_{L^2(\Gamma(t))}^2 - c_{P,2}(t) \big(\int_{\Gamma(t)} u \dif s\big)^2$ holds. Using smoothness assumptions (to have uniformity in $t$ of the constants) and suitable perturbation arguments  (cf. \cite[Theorem 5.16]{Diss_Sass}) one obtains, for sufficiently smooth functions $u$ on $S_h^n$, $n=1,\ldots,N$,
	\begin{equation}
		\norm{\nabla_{\Gamma_h}u}_{L^2(S_h^n)}^2\geq C_{P,1}\norm{u}_{L^2(S_h^n)}^2-  C_{P,2}\int_{t_{n-1}}^{t_n}\big(\overline{u}_h^n\big)^2 \dif t,\label{PCdisco}
	\end{equation}
	with strictly positive and uniform (in the discretization parameters) constants $C_{P,1}$, $C_{P,2}$. In  the analysis we have to control the second term on the right-hand side in \eqref{PCdisco}. For that we need the stabilization term $s_2(\cdot,\cdot)$. 

\section{Discretization error analysis}\label{consistency}
In this section we present  an error analysis for the space-time discretization method \eqref{discreteproblem}.
Optimal error bounds for the case of bi-linear finite element approximations are presented in \cite{Diss_Sass}. The analysis in this section applies also to the case of higher order (iso)parametric finite element spaces.
It is well-known that in the analysis of finite element discretization methods on stationary or evolving surfaces the treatment of the geometric error, i.e., the error caused by the numerical approximation of the surface, is a key point. Typically the analysis of the corresponding consistency error is rather technical. In the setting of the space-time finite element method \eqref{discreteproblem} one has to control the difference between quantities  on $S$ and the corresponding ones on $S_h$.  For this the results derived in the Sections~\ref{section_basic_geometry_est} and \ref{secGeometry} are essential ingredients.

We start with assumptions that simplify both the analysis and the presentation. We restrict to the \emph{iso}parametric case, in the sense that the polynomial degrees used in the geometry approximation (parametric mapping) are the same as the ones used in the finite element spaces, i.e., $k_{g,s}=k_s$ and $k_{g,q}=k_{q}$. To further simplify the presentation we take $k_s=k_q$. We avoid space-time grid anisotropies by assuming
\[ \Delta t \sim h.
\]
In the remainder of this section, as discretization parameters we use the mesh size $h$ ($\sim \Delta t$) and the polynomial degree $k_s$ ($=k_{g,s}=k_{g,q}=k_{q}$). For the isoparametric finite element space we use the simplified notation $ V_{h,\Theta}^{k_\spa}=V_{h,\Theta}^{k_\spa, k_\ti}$. 

In the subsections below we treat different components of the error analysis. These are the usual ones for an inconsistent finite element method, namely a stability estimate, consistency error estimates and interpolation error estimates. In Section~\ref{secStrang} we derive a stability result and formulate a Strang lemma.
The stability analysis is fairly straightforward, 
because the bilinear form in the discretization \eqref{discreteproblem} is antisymmetric with respect to the material derivative, which almost immediately leads to a stable method.    In Section~\ref{secconsistency} we analyze consistency errors. In that section many results presented in the Sections~\ref{section_basic_geometry_est}-\ref{secGeometry} will be used. Finally in  Section~\ref{secinterpol} we briefly address interpolation error estimates and derive the main result Theorem~\ref{mainthm}, in which an optimal discretization error bound for the method \eqref{discreteproblem} is presented.

In the error analysis we use the natural norm
 \begin{align}
 	\begin{aligned}\label{energy_norm_ho_def}
 		\triplenorm u \triplenorm_{h,\Theta}^2\coloneqq & \norm{u}^2_{L^2(S_h)}+ \norm{\nabla_{\Gamma_h}u}_{L^2(S_h)}^2 
 		+ \xi_1\norm{\n_h\cdot \nabla u}_{L^2(Q^{S}_\Theta)}^2 \\
 		&  +\norm{u_-^N}^2_{L^2(\Gamma_{h}^N(t_N))}+ {\sum_{n=1}^N}\norm{[u]_h^{n-1}}^2_{L^2(\Gamma_{h}^n(t_{n-1}))}.
 	\end{aligned}
\end{align}

\subsection{Stability estimate and Strang lemma} \label{secStrang}
We derive an ellipticity estimate for the bilinear form $B_h^{\stab}$ in the $ \triplenorm\cdot \triplenorm_{h,\Theta}$-norm. Note that this norm does not contain time derivatives. When considering $B_h^{\stab}(u,u)$, $u\in V_{h,\Theta}^{k_s}$, the terms with the time derivative vanish. Hence, in the stability analysis we do not need partial integration and we do not need bounds for the (geometric) error terms that appear in the formula \eqref{PI-unschoen}.  However, we do have to control the geometric errors in the time slab boundary terms in \eqref{discreteformendetail}. 
 
We start with an assumption that we need in the stability analysis.
\begin{assumption}\label[assumption]{ass_stability}
	We assume that there exists a constant $c_{\Div}>0$ such that
	\begin{equation}
		\frac{\div_{\Gamma_h}\mathbf{w}}{\alpha_h}+c_\alpha\mu_d C_{P,1}\geq c_{\div}\quad \text{a.e. on } S_h ,\label{divforderungdiscrete}
	\end{equation}
	where $C_{P,1}$ is the constant of the Poincar\'{e} inequality \cref{PCdisco} and $c_\alpha\coloneqq \norm{\alpha_h}^{-1}_{L^\infty(S_h)}$. 
\end{assumption}

Note that for convection-diffusion problems assumptions of this type are often used in the analysis of finite element methods. 
	For  the stabilization parameter  $\xi_2$ used in $s_2(\cdot,\cdot)$ we assume
	\begin{equation}
	\frac{\mu_d C_{P,2}c_\alpha}{2}	 \leq \xi_2 \lesssim h^{-1},\label{stabiparameterbedingung}
	\end{equation}
	with $C_{P,2}$ from the Poincar\'{e} inequality \cref{PCdisco}.
 We now introduce an estimate that is crucial in the stability analysis of trace and cut finite element methods. This estimate essentially bounds the $L^2$-norm of finite element functions on a small strip around the discrete surface by the sum of the $L^2$-norm on the surface and the $L^2$-norm of the normal derivative in the strip.  For stationary surfaces such estimates are known in the literature, also for higher order finite element spaces, e.g. \cite{grande2018analysis}. For the discrete space-time surface that we consider in this paper and the case $k_s=1$  (bi-linear finite elements) a proof is given in \cite{Diss_Sass}.   We think that the arguments used in this literature can be extended to the  higher order space-time case. A rigorous proof, however, requires a detailed technical analysis that we do not present here. Without proof we introduce the following claim. 
\begin{proposition}\label[assumption]{ass_ho_tech}
	For $v_h\in V_{h,\Theta}^{k_s}$ the following uniform estimate holds:
	\begin{equation}
		\norm{v_h}^2_{L^2(Q^{S}_{\Theta})}\lesssim h \left(\norm{v_h}^2_{L^2(S_h)}+s_1(v_h,v_h)\right).\label{vh_tech}
	\end{equation}
\end{proposition}
The estimate \cref{vh_tech} can be used to show that \cref{energy_norm_ho_def}  defines a norm.
Using standard finite element trace estimates and \cref{vh_tech}  one obtains
		\begin{align}
		h\sum_{F\in\mathcal{F}_I} \norm{v_h}^2_{L^2(F)}&\lesssim\triplenorm v_h\triplenorm_{h,\Theta}^2,\label{CustomInvIneq2}\\
		h\sum_{n=1}^N \norm{v_{h,-}^n}^2_{L^2(\Gamma_h^n(t_n))}& 
		\lesssim\triplenorm v_h\triplenorm_{h,\Theta}^2,\label{CustomInvIneq_top}\\
		h\sum_{n=1}^N\norm{v_{h,+}^{n-1}}^2_{L^2(\Gamma_h^n(t_{n-1})} & 
		\lesssim\triplenorm v_h\triplenorm_{h,\Theta}^2\label{CustomInvIneq_bot}.
	\end{align}

\begin{theorem}\label{theo_stabi}
 There exists  $c_s> 0$ such that for sufficiently small $h >0$ the following ellipticity estimate holds:
	\begin{equation}
		\inf_{u\in V_{h,\Theta}^{k_s}} \frac{B_h^{\stab}(u,u)}{ \triplenorm u \triplenorm_{h,\Theta}^2 }\geq c_s.\label{infsupzz}
	\end{equation}
\end{theorem}
\begin{proof}
	For $k_s=1$ a proof is given in \cite[Theorem 5.19]{Diss_Sass}. We assume $k_s>1$. We note that for $k_s>1$ the proof is simpler than for $k_s=1$ because we can benefit from smaller perturbation terms $\|R-1\|_{L^\infty(S_h)} \lesssim h^{k_s}$, cf. Lemma~\ref{l4}. 
	 We take an arbitrary $u\in V_{h,\Theta}^{k_s}, u\neq 0$.   By definition we have
	\begin{align}
			B_h^{\stab}(u,{u}) &= \underbrace{ \sum_{n=1}^N \int_{S_h^n}\frac{1}{2\alpha_h}\left(u^2\div_{\Gamma_h}\mathbf{w}
			+2\mu_d \nabla_{\Gamma_h}u\cdot \nabla_{\Gamma_h}u\right)\dif \sigma_h + s_2(u,u)}_{F_1} \nonumber\\
			& +
			\sum_{n=1}^N\bigg[ \frac12\left(R_-^n u_-^n,u_-^n\right)_{\Gamma_h^n(t_n)}+\frac12\left(R_+^{n-1} u_+^{n-1},u_+^{n-1}\right)_{\Gamma_h^n(t_{n-1})} \label{first_step_stabi} \\
			& \qquad -\left( R_+^{n-1} u^{n-1}_-\circ \Theta_h^{n-1}\circ(\Theta_h^{n})^{-1},u_+^{n-1}\right)_{\Gamma_h^n(t_{n-1})}\bigg] +  s_1(u,u) \nonumber \\
			& =: F_1+ F_2 + s_1(u,u). \nonumber
	\end{align}
	For the term $F_1$ we use the Poincar\'{e} inequality \cref{PCdisco} and the choice of $\xi_2$  
	\eqref{stabiparameterbedingung} in the stabilization 
	term $s_2(u,u)$:
	\begin{align*}
		\begin{aligned}
			&\sum_{n=1}^N\bigg[\int_{S_h^n}\frac{1}{2\alpha_h}\left(u^2\div_{\Gamma_h}\mathbf{w}+ 2\mu_d \nabla_{\Gamma_h}u\cdot \nabla_{\Gamma_h}u\right)\dif \sigma_h \bigg]+ s_2(u,u)\\
			&\qquad \geq \sum_{n=1}^N \bigg[ \frac{c_\alpha\mu_d}{2}\norm{\nabla_{\Gamma_h}u}_{L^2(S_h^n)}^2+ \int_{S_h^n} \frac12 \big(\frac{1}{\alpha_h}\div_{\Gamma_h}\mathbf{w}+ c_\alpha\mu_d C_{P,1}\big) u^2 \dif \sigma_h \bigg]
	  \end{aligned}
	  \end{align*}
	  Using the assumption \eqref{divforderungdiscrete} we then obtain
	   \begin{equation}
	    \label{ell2.teil}
	  F_1 \geq \tfrac12
		\min\{\mu_dc_\alpha, c_{\div}\} \big( \norm{u}^2_{L^2(S_h)}+\norm{\nabla_{\Gamma_h}u}^2_{L^2(S_h)}\big).
	\end{equation}
	We now treat the term $F_2$. We use the notation $ T_\Theta^{n-1}:=\Theta_h^{n-1}\circ(\Theta_h^{n})^{-1} : \Gamma_h^n(t_{n-1}) \to \Gamma_h^{n-1}(t_{n-1})$. Let  $J_\Theta^{n-1} :=\det ( DT_\Theta^{n-1})$ denote the corresponding Jacobian determinant and
	\[
	 \tilde F_2:=\sum_{n=1}^N\bigg[ \frac12\norm{u_-^n}^2_{L^2(\Gamma_h^n(t_n))} + \frac12\norm{u_+^{n-1}}^2_{L^2(\Gamma_h^n(t_{n-1}))}-  \left( u^{n-1}_-\circ T_\Theta^{n-1} ,u_+^{n-1}\right)_{\Gamma_h^n(t_{n-1})}\bigg]
	\]
the term $F_2$ with all $R$-terms replaced by 1. Using $\|R-1\|_{L^\infty(S_h)} \lesssim h^{2}$ and \eqref{CustomInvIneq_top}-\eqref{CustomInvIneq_bot} we get  
\begin{equation} \label{F2}
 |F_2 -\tilde F_2 | \lesssim  h \triplenorm u \triplenorm_{h,\Theta}^2.
 \end{equation}
We further consider $\tilde F_2$. Recall the jump term \eqref{discrete_jumps} across $\Gamma_h(t_{n-1})$:  $[u]_h^{n-1}=u_+^{n-1}-u_-^{n-1} \circ T_\Theta^{n-1}$. Hence,
\begin{equation} \label{po} \begin{split}
  & -  \left( u^{n-1}_-\circ T_\Theta^{n-1} ,u_+^{n-1}\right)_{\Gamma_h^n(t_{n-1})}   =  - \frac12 \|u^{n-1}_-\circ T_\Theta^{n-1}\|_{L^2(\Gamma_h^n(t_{n-1}))}^2  \\ &\qquad \qquad\qquad \qquad - \frac12 \|u_+^{n-1}\|_{L^2(\Gamma_h^n(t_{n-1}))}^2 
 +  \frac12 \|[u]_h^{n-1}\|_{L^2(\Gamma_h^n(t_{n-1}))}^2.
\end{split} \end{equation}
We now relate $\|u^{n-1}_-\circ T_\Theta^{n-1}\|_{L^2(\Gamma_h^n(t_{n-1}))}^2$ to $\|u^{n-1}_-\|_{L^2(\Gamma_h^{n-1}(t_{n-1}))}^2$:
\[																																
 \|u^{n-1}_-\|_{L^2(\Gamma_h^{n-1}(t_{n-1}))}^2 = \int_{ \Gamma_h^{n-1}(t_{n-1}))} \big(u^{n-1}_-\big)^2 \, ds = 
 \int_{\Gamma_h^n(t_{n-1}))} \big(u^{n-1}_- \circ T_\Theta^{n-1} \big)^2  |J_\Theta^{n-1}|\, ds.
 \]
Using \cref{Dtheta_psi_est} we get
\begin{align*}
	&\max_{K\in\Tau_n^\Gamma}\norm{D\Theta_h^n(\cdot, t_n)D\Theta_h^{n+1}(\cdot, t_n)^{-1}- I}_{L^\infty(K)}\\
	\lesssim &\max_{K\in\Tau_n^\Gamma}\norm{D\Theta_h^n(\cdot, t_n)-D\Theta_h^{n+1}(\cdot, t_n)}_{L^\infty(K)}\\
	=&\max_{K\in\Tau_n^\Gamma}\norm{(D\Theta_h^n(\cdot, t_n)-D\Psi(\cdot, t_n)) + (D\Psi(\cdot, t_n)-D\Theta_h^{n+1}(\cdot, t_n)}_{L^\infty(K)} \lesssim h^{k_s} \lesssim
	 h^2,
\end{align*}
and thus $|J_\Theta^{n-1}-1| \lesssim  h^2$. 
This yields (for $h$ small enough) 
\[  \left| \|u^{n-1}_-\circ T_\Theta^{n-1}\|_{L^2(\Gamma_h^n(t_{n-1}))}^2 - \|u^{n-1}_-\|_{L^2(\Gamma_h^{n-1}(t_{n-1}))}^2\right| \lesssim h^2 \|u^{n-1}_-\|_{L^2(\Gamma_h^{n-1}(t_{n-1}))}^2.
\]
 We use this in \eqref{po} and substitute in the expression for $\tilde F_2$. This yields
\[
 \tilde F_2 \geq \frac12 \|u^{N}_-\|_{L^2(\Gamma_h^{N}(t_{N}))}^2 +  \frac12 \sum_{n=1}^N  \|[u]_h^{n-1}\|_{L^2(\Gamma_h^n(t_{n-1}))}^2 -  c h^2 \sum_{n=1}^N\|u^{n-1}_-\|_{L^2(\Gamma_h^{n-1}(t_{n-1}))}^2,
\]
with a suitable $c>0$.
Combining this with \eqref{CustomInvIneq_top} and \eqref{F2} we get  
\begin{equation} \label{PO}
 F_2 \geq \frac12 \|u^{N}_-\|_{L^2(\Gamma_h^{N}(t_{N}))}^2 +  \frac12 \sum_{n=1}^N  \|[u]_h^{n-1}\|_{L^2(\Gamma_h^n(t_{n-1}))}^2 - c  h \triplenorm u \triplenorm_{h,\Theta}^2.
\end{equation}
Combining the estimates for $F_1$ and $F_2$ and noting that $s_1(u,u)$ forms part of $\triplenorm u \triplenorm_{h,\Theta}^2$ we obtain
\[
 B_h^{\rm stab} (u,u) \geq \tfrac12 \min\{\mu_dc_\alpha, c_{\div},1\} \triplenorm u \triplenorm_{h,\Theta}^2 - c h\triplenorm u \triplenorm_{h,\Theta}^2.
\]
Hence, for $h >0$ sufficiently small we obtain the ellipticity estimate \eqref{infsupzz}.
\end{proof}

For the formulation of the Strang lemma we use  an interpolation operator in the volume space-time finite element space  denoted  by ${I_\Theta^{k_\spa}:C^0(Q_\Theta^S)\rightarrow V_{h,\Theta}^{k_\spa}}$.   
\begin{lemma} \label{LemmaStrang}
	Let $u\in H^2(S)$ be the solution of \cref{prob_cont} and $u_h$ the solution of \Cref{discreteproblem}. The following estimate holds.
	\begin{align}
		\begin{aligned}
			\triplenorm u^e-u_h \triplenorm_{h,\Theta}&\leq \triplenorm u^e-I_\Theta^{k_\spa}(u^e)\triplenorm_{h,\Theta} + c \sup_{v_h\in V_{h,\Theta}^{k_\spa}} \frac{B_h^{\stab}(u^e-I_\Theta^{k_\spa}(u^e),v_h)}{\triplenorm v_h \triplenorm_{h,\Theta}}  \\
			&\qquad+   c \sup_{v_h\in  V_{h,\Theta}^{k_\spa}} \frac{B_h^{\stab}(u^e-u_h,v_h)}{\triplenorm v_h \triplenorm_{h,\Theta}}.\label{strangFormula}
		\end{aligned}
	\end{align}
\end{lemma}

\begin{proof}
	This follows using a standard argument. By the triangle inequality and the ellipticity estimate \cref{infsupzz} we have
	\begin{align*}
		\triplenorm u^e-u_h \triplenorm_{h,\Theta}
		& \leq  \triplenorm u^e-I_\Theta^{k_\spa}(u^e)\triplenorm_{h,\Theta} + c \sup_{v_h\in V_{h,\Theta}^{k_\spa}}\frac{B_h^{\stab}(u_h-I_\Theta^{k_\spa}(u^e),v_h)}{\triplenorm v_h \triplenorm_{h,\Theta}}\\
		& \leq  \triplenorm u^e-I_\Theta^{k_\spa}(u^e)\triplenorm_{h,\Theta}  \\ & \quad + c\left( \sup_{v_h\in V_{h,\Theta}^{k_\spa}} \frac{B_h^{\stab}(u^e-I_\Theta^{k_\spa}(u^e),v_h)}{\triplenorm v_h \triplenorm_{h,\Theta}} +    \sup_{v_h\in  V_{h,\Theta}^{k_\spa}} \frac{B_h^{\stab}(u^e-u_h,v_h)}{\triplenorm v_h \triplenorm_{h,\Theta}}\right).
	\end{align*}
	\end{proof}
	\ \\
Estimates for the first term in the bound in \eqref{strangFormula} can be based on fairly straightforward interpolation error bounds. Such interpolation error bounds combined with a suitable continuity argument lead to an estimate for the second term on the right-hand side in \eqref{strangFormula}, too. These estimates are further discussed in   Section~\ref{secinterpol}. In Section~\ref{secconsistency} we derive an estimate of the remaining third term, which measures the inconsistency due to geometric errors. Note that in case of an exact surface, i.e., $S_h=S$ and the finite element space $V_{h,\Theta}^{k_s}$ replaced by $V_h^{k_s}$ this term vanishes due to Galerkin orthogonality.

\subsection{Consistency error analysis}
 \label{secconsistency} 
 The most technical part of the discretization error analysis is the derivation of bounds for the consistency error (third term of the right-hand side in \eqref{strangFormula}), which measures the effect of the geometry approximation. In this section we derive such bounds.  We will use key results derived in the Sections~\ref{section_basic_geometry_est} and \ref{secGeometry},  e.g., the partial integration rule \eqref{PI-unschoen} and the geometric estimates in the Lemmas~\ref{l_mu}, \ref{l3} and \ref{l4}.   
 
We start with two estimates that we use in the proof of Theorem~\ref{thmconsisteny} below.
\begin{lemma}\label{lemma_divs}
 The following uniform estimates hold
 \begin{align} \label{div2A}
 \norm{\div_{S_h}\big(\frac{1}{\alpha_h}\bP_{S_h} \w_S\big) - \big(\div_S\frac{1}{\alpha}\w_S\big)^e}_{L^\infty(S_h)} &\lesssim h^{k_{\spa}},\\
 \norm{\Div_{\Gamma_h}(\bw^e)-(\Div_{\Gamma}\bw)^e}_{L^\infty(S_h)}&\lesssim h^{k_\spa}.\label{div3}
\end{align}
\end{lemma}
\begin{proof}
 We write
 \[ \begin{split}
 & \div_{S_h}\big(\frac{1}{\alpha_h}\bP_{S_h} \w_S\big) - \big(\div_S\frac{1}{\alpha}\w_S\big)^e = \\
  & \Big(\div_{S_h}\big(\frac{1}{\alpha_h}\bP_{S_h} \w_S\big) - \div_{S_h}\frac{1}{\alpha^e}\w_S^e\Big) + \Big( \div_{S_h}\frac{1}{\alpha^e}\w_S^e - \big(\div_S\frac{1}{\alpha}\w_S\big)^e\Big)=I+II. 
 \end{split} \]
For the term $II$ we use \eqref{estdiv} with $\bv_h=\frac{1}{\alpha^e}\w_S^e$, i.e., $\bv_h^l=\frac{1}{\alpha}\bw_S$, and thus we get
\[
 \norm{\div_{S_h}\frac{1}{\alpha^e}\w_S^e - \big(\div_S\frac{1}{\alpha}\w_S\big)^e}_{L^\infty(S_h)} \lesssim h^{k_{\spa}}.
\]
For term $I$ we obtain, using the product rule, \eqref{alpha_grad_est} and $\norm{\div_{S_h}( \bw_S -\w_S^e)}_{L^\infty(S_h)}  \lesssim h^{k_{\spa}}$:
\begin{align*}
  &\norm{\div_{S_h}\big(\frac{1}{\alpha_h}\bP_{S_h} \w_S\big) - \div_{S_h}\frac{1}{\alpha^e}\w_S^e}_{L^\infty(S_h)}\\
  \lesssim & \norm{\div_{S_h}\big(\frac{1}{\alpha^e}( \bP_{S_h} \bw_S -\w_S^e)\big)}_{L^\infty(S_h)} + h^{k_{\spa}}   \lesssim \norm{\div_{S_h}( \bn_{S_h} \bn_{S_h}\cdot \bw_S) }_{L^\infty(S_h)}+ h^{k_{\spa}}  \\
  \lesssim &\norm{\bn_{S_h} \cdot \bw_S}_{L^\infty(S_h)}  \norm{\div_{S_h} \bn_{S_h}}_{L^\infty(S_h)}+ h^{k_{\spa}}  \lesssim h^{k_{\spa}}.
\end{align*}
\textcolor{black}{The second last inequality follows from the product rule for the divergence operator and by noting that $\n_{S_h}^T\nabla_{S_h}(\n_{S_h}\cdot \w_S)=0$}.  
In the last inequality we used $\bn_{S_h} \cdot \bw_S = \bP_S^e \bn_{S_h} \cdot \bw_S $ and the estimate \eqref{st_normal_est}. 
A proof of  \cref{div3} follows with similar arguments as in the proof of Corollary~\ref{CorH}, using the transformation formula \eqref{trafo_demlow_grad}.
\end{proof}
 \ \\
	\begin{theorem} \label{thmconsisteny}
	Let $u\in H^2(S)$, $u\in H^1(\Gamma(t))$ for all $t\in[0,T]$, be the solution of \cref{prob_cont} and $u_h$ the solution of \cref{discreteproblem} with $f_h$ that satisfies \eqref{f_approx}. Then we have the bound
	\begin{equation}
		\sup_{v_h\in V_{h,\Theta}^{k_\spa}} \frac{\abs{B_h^{\stab}(u^e-u_h,v_h)}}{\triplenorm v_h \triplenorm_{h,\Theta}}\lesssim h^{k_{\spa}} \big(\norm{u}_{H^2(S)}+h\norm{f}_{L^2(S)}\big).\label{konsistenz}
	\end{equation}
\end{theorem}

\begin{proof} 
Let $v_h\in V_{h,\Theta}^{k_\spa}$ be arbitrary. First note that using $ \mu_h \dif s_h= \dif s \circ \bp$ we get the relation  
\begin{align}
	\begin{aligned}
		B_h^{\stab}(u_h,v_h)= \sum_{n=1}^N\int_{t_{n-1}}^{t_n}\int_{\Gamma_h^n(t)}(f_h-\mu_hf^e) v_h\dif s_h\dif t+ \int_0^T\int_{\Gamma(t)} f v_h^l \dif s\dif t.\label{consis_start_1}
	\end{aligned}
\end{align}
The function $u$ solves the continuous problem \cref{prob_cont}, i.e. we have 
\begin{equation*}
	\int_0^T\int_{\Gamma(t)} f v_h^l \dif s\dif t=\sum_{n=1}^N\left[\int_{S^n}\frac{1}{\alpha}\left(\dot{u} v_h^l  + uv_h^l\div_{\Gamma}\w+\mu_d\nabla_{\Gamma}u\cdot \nabla_{\Gamma}v_h^l\right)\dif \sigma\right].
\end{equation*}
Note that $u\in C^0(S)$, hence the jump term in \cref{prob_cont} vanishes. Note that $\dot{u} = \w_S \cdot \nabla_S u$.  We use the  results \eqref{helptrans}, \cref{Trafo_With_A_h} and \cref{2.2.19} to transform back to the discrete manifold $S_h$, i.e.,
\begin{equation}\label{consis_start_2}
	\begin{split}
		 &\int_0^T\int_{\Gamma(t)} f v_h^l \dif s\dif t =\sum_{n=1}^N\Bigg[ \int_{S_h^n}\frac{\mu_h}{\sqrt{1+V_h^2}} (\w_S^e \cdot\tilde{\gd}^{-1}\gp_0\nabla_{S_h}u^e) v_h\\
		&\qquad+ \frac{\mu_h}{\sqrt{1+V_h^2}} u^ev_h (\div_{\Gamma} \w)^e+\mu_d\ga_h \nabla_{\Gamma_h}u^e\cdot \nabla_{\Gamma_h}v_h\dif \sigma_h \Bigg].
	\end{split}
\end{equation}
On the other hand, for the (extended) exact solution $u^e$ substituted in the discrete bilinear form \Cref{stabbilform} we obtain
\begin{align*}
	&B_h^{\stab}(u^e,v_h)=\sum_{n=1}^N\bigg[\int_{S_h^n}\frac{1}{2\alpha_h}\big(v_h \mathring{u^e}
	 -u^e \mathring{v}_h+u^ev_h\div_{\Gamma_h}\w+ 2 \mu_d\nabla_{\Gamma_h}u^e\cdot \nabla_{\Gamma_h}v_h \big) \dif \sigma_h\\\
	& \qquad +
	\frac{1}{2}\int_{\Gamma_h^n(t_n)}u^e v_{h,-}^n R_-^{n}\dif s_h +\frac12 \int_{\Gamma_h^n(t_{n-1})} u^e v_{h,+}^{n-1}R_+^{n-1}\dif s_h\\
	&\qquad - \int_{\Gamma_h^n(t_{n-1})} (u^e \circ\Theta_h^{n-1}\circ (\Theta_h^n)^{-1})v_{h,+}^{n-1}R_+^{n-1} \dif s_h  \bigg] + s_1(u^e,v_h) + s_2(u^e,v_h).
\end{align*}
Partial integration \cref{PI-unschoen} leads to 
\begin{align}
			&B_h^{\stab}(u^e,v_h) =\sum_{n=1}^N\bigg[ 
		\int_{S_h^n}\frac{1}{ \alpha_h}\big(  \mathring{u^e} v_h+ \frac12 u^ev_h\div_{\Gamma_h}\w +
		  \mu_d \nabla_{\Gamma_h}u^e\cdot\nabla_{\Gamma_h}v_h \big)\dif\sigma_h  \nonumber \\ & \qquad + \int_{\Gamma_h^n(t_{n-1})}[u^e]^{n-1}_h v_{h,+}^{n-1}R_+^{n-1}\dif s_h+\frac12\sum_{K_S\in  \mathcal{T}_{S_h^n}}\int_{K_S} u^ev_h\div_{S_h}\left( \frac{1}{\alpha_h} \gp_{S_h}\w_S\right)\dif \sigma_h \nonumber \\
		&\qquad- \frac12\sum_{F\in \mathcal{F}_I^n}\int_{F} u^ev_h \w_S\cdot \left[\alpha_h^{-1}\right]_{\bnu}\dif F\bigg]+ s_1(u^e,v_h) + s_2(u^e,v_h).\label{pi_used_consis}
\end{align}
We combine \cref{consis_start_1,consis_start_2,pi_used_consis} and obtain, using $\mathring{u^e} = \bw_S \cdot \nabla_{S_h}u^e$,
\begin{align}
	&B_h^{\stab}(u^e-u_h,v_h) =
	\int_{S_h}\left(\frac{1}{\alpha_h}\w_S^T-\frac{\mu_h}{\sqrt{1+V_h^2}} \w_S^e \cdot\tilde{\gd}^{-1}\gp_0\right) \nabla_{S_h}u^e  v_h\dif \sigma_h  \label{T1}\\
	&  + \underbrace{\int_{S_h} \frac{1}{2\alpha_h}u^e v_h\div_{\Gamma_h}\w
	-\frac{\mu_h}{\sqrt{1+V_h^2}} u^ev_h (\div_{\Gamma} \w)^e\dif \sigma_h}_{I}  \nonumber \\
	& +\mu_d\int_{S_h}\left(\frac{1}{\alpha_h}\mathbf{I}-\ga_h\right)\nabla_{\Gamma_h}u^e\cdot \nabla_{\Gamma_h}v_h\dif \sigma_h \label{T2} \\ 
	& +\sum_{n=1}^N\bigg[ \int_{\Gamma_h^n(t_{n-1})}[u^e]^{n-1}_h v_{h,+}^{n-1}R_+^{n-1}\dif s_h\bigg] \label{T3} \\ & +\underbrace{\frac12\sum_{K_S\in  \mathcal{T}_{S_h}}\int_{K_S} u^ev_h\div_{S_h}\left( \frac{1}{\alpha_h} \gp_{S_h}\w_S\right)\dif \sigma_h}_{II}
	 - \underbrace{\frac12\sum_{F\in \mathcal{F}_I}\int_{F} u^ev_h \w_S\cdot {\color{black}\bP_S^e}\left[\alpha_h^{-1}\right]_{\bnu}\dif F}_{III}  \nonumber\\ &
	 + \int_{S_h}\frac{(f_h-\mu_hf^e) v_h}{\sqrt{1+V_h^2}}\dif \sigma_h+ s_1(u^e,v_h) + s_2(u^e,v_h).\label{T4}
\end{align}
We now derive bounds for these terms using results from Sections~\ref{section_basic_geometry_est} and \ref{secGeometry}.
We first consider the term on the right-hand side in \eqref{T1}, denoted by $T_1$. Using the smoothness of $\bw$, \cref{einminusmu_Space}, \cref{alpha_h_sqrtvh} and \cref{ADD} we get 
\[ \big\|\big(\frac{1}{\alpha_h}\w_S^T-\frac{\mu_h}{\sqrt{1+V_h^2}} \w_S^T \tilde{\gd}^{-1}\gp_0\big) \bP_{S_h}\big\|_{L^\infty(S_h)} \lesssim \big\|\bP_{S_h} - \tilde{\gd}^{-1}\gp_0\big\|_{L^\infty(S_h)} +h^{k_s} \lesssim h^{k_s}.
\]
From this it follows that $|T_1| \lesssim h^{k_{s}} \norm{v_h}_{L^2(S_h)}\norm{u}_{H^1(S)}$. 
We now estimate the term in \eqref{T2}, denoted by $T_2$. 
 Recall the definition of $\ga_h$ in \cref{ah_def}. We again use \cref{einminusmu_Space}, \cref{alpha_h_sqrtvh}. Also note that  ${\left(\bI-\delta \gh\right)^{-1}=\bI+\mathcal{O}(h^{k_{\spa}+1})}$ and  ${\gp_h\tilde{\gp}_h=\gp_h}$. With these results we obtain
\[
	\big\|\big(\frac{1}{\alpha_h} \bI -\ga_h\big)\gp_h)\big\|_{L^\infty(S_h)}\lesssim  \big\|\gp_h-\tilde{\gp}_h\big\|_{L^\infty(S_h)}+h^{k_{\spa}}\lesssim h^{k_{\spa}},
\]
and with this we get $|T_2| \lesssim h^{k_{\spa}} \norm{v_h}_{L^2(S_h)}\norm{u}_{H^1(S)}$.
We consider the term in \eqref{T3}, denoted by $T_3$. 
Using the continuity of $u^e$ we get, on $\Gamma_h^n(t_{n-1})$:
\begin{equation} \label{A1neu}
  [u^e]_h^{n-1}= u^e\big({\rm id} -\Theta_h^{n-1}\circ (\Theta_h^n)^{-1}\big)= u^e \big (\Theta_h^n-\Theta_h^{n-1}\big)(\Theta_h^n)^{-1}. 
\end{equation}
Using \eqref{theta_psi_est} and a Sobolev embedding estimate  we obtain $\big\|[u^e]_h^{n-1}\big\|_{L^\infty(\Gamma_h^n(t_{n-1}))} \lesssim h^{k_{\spa}+1} \|u\|_{H^2(S)}$. Combining this with \eqref{CustomInvIneq_bot} yields
\begin{equation} \label{A2neu} \begin{split}
  |T_3|  & \lesssim 
 h^{k_{\spa}+1} \|u\|_{H^2(S)}\sum_{n=1}^N\norm{v_{h,+}^{n-1}}_{L^2(\Gamma^n_h(t_{n-1}))}  \\ & \lesssim h^{k_{\spa}+\frac12}
  \big(\sum_{n=1}^N\norm{v_{h,+}^{n-1}}_{L^2(\Gamma^n_h(t_{n-1}))}^2 \big)^\frac12\|u\|_{H^2(S)}   \lesssim
  h^{k_{\spa}} \triplenorm v_h\triplenorm_{h,\Theta} \|u\|_{H^2(S)}.
  \end{split}
  \end{equation}
We combine the terms $(I)$ and $(II)$. First note that using \cref{div3}, \cref{einminusmu_Space}, \cref{alpha_h_sqrtvh} and \cref{alpha_est}
we get 
\[
 \big|(I) + \int_{S_h} \frac{1}{2 \alpha^e}u^e v_h (\div_\Gamma \w)^e \dif \sigma_h \big| \lesssim h^{k_{\spa}}\norm{u}_{L^2(S)}\norm{v_h}_{L^2(S_h)}.
\]
Using \cref{div2A} and \eqref{elneu} yields
\begin{equation*}
	\abs{(II) -\int_{S_h}\frac{1}{2\alpha^e} u^ev_h (\div_\Gamma \w)^e \dif \sigma_h } \lesssim h^{k_{\spa}}\norm{u}_{L^2(S)}\norm{v_h}_{L^2(S_h)}.
\end{equation*}
Combining this we obtain $|(I) + (II)| \lesssim h^{k_{\spa}}\norm{u}_{L^2(S)}\norm{v_h}_{L^2(S_h)}$.

We turn to the bound for $(III)$. We need the following Hansbo trace type inequality. For an arbitrary $F\in \mathcal{F}_I$ we take a corresponding four-dimensional space-time prism $P$ of the outer triangulation  $P\in Q_h^{S}$ with $P\cap F\neq \emptyset$. We then have for $w \in H^2(P)$
\begin{equation} \begin{split}
	\norm{w}_{L^2(F)}^2 & \lesssim h^{-1}\norm{w}_{L^2(
\partial P)}^2+h\norm{w}_{H^1(
\partial P)}^2  \\ & \lesssim h^{-2}\norm{w}_{L^2(P)}^2+ \norm{w}_{H^1(
P)}^2 + h^2\norm{w}_{H^2(
P)}^2, \label{hansbo_var}
\end{split} \end{equation}
which follows from \cite[Lemma 4.2]{Hansbo03}.
Note that for $P \in Q_h^{S}$ we have $\|\delta\|_{L^\infty(P)} \lesssim h$ and using $\bn \cdot \nabla u^e = 0 $ on $P$ we get \cite[Section 6.1]{olshanskii2014error} $\norm{u^e}^2_{H^2(P)}\lesssim h\norm{u}^2_{H^2(S \cap P)}$. 
With this, the bound of the conormal jumps in \cref{alpha_h_good_jump}, the trace inequality \eqref{hansbo_var} and the face integral bound \cref{CustomInvIneq2} we obtain:
	\begin{align*}
		\abs{(III)}&\lesssim h^{k_{\spa}+1}\big(\sum_{P\in Q_h^{S}}h^{-2}\norm{u^e}^2_{L^2(P)}+2\norm{u^e}^2_{H^1(P)}+h^{2}\norm{u^e}^2_{H^2(P)}\big)^{\frac12} \big(\sum_{F\in \mathcal{F}_I}\norm{v_h}^2_{L^2(F)}\big)^{\frac12}\\
		&\lesssim h^{k_{\spa}+1}\big(\sum_{P\in Q_h^{S}}h^{-2}\norm{u^e}_{H^2(P)}^2\big)^{\frac12} \big(\sum_{F\in \mathcal{F}_I}\norm{v_h}^2_{L^2(F)}\big)^{\frac12}\\
		&\lesssim h^{k_{\spa}+1}h^{-\frac12}\norm{u}_{H^2(S)} h^{-\frac12} \triplenorm v_h\triplenorm_{h,\Theta}\lesssim h^{k_{\spa}}  \norm{u}_{H^2(S)} \triplenorm v_h\triplenorm_{h,\Theta}. 
	\end{align*}
Finally we consider the terms in \eqref{T4}.   For the first term in \eqref{T4} we get, using the data accuracy assumption
\cref{f_approx}:
\begin{equation}
	\abs{ \int_{S_h}\frac{(f_h-\mu_hf^e) v_h}{\sqrt{1+V_h^2}}\dif \sigma_h}\lesssim h^{k_{\spa}+1} \norm{f}_{L^2(S)}\norm{v_h}_{L^2(S_h)}.\label{f9_estimate}
\end{equation}
For the normal gradient stabilization term we use $\n \cdot \nabla u^e =0$, \cref{spa_normal_est} and $\norm{u^e}^2_{H^2(Q^S)}\lesssim h\norm{u}^2_{H^2(S)}$:
\begin{align}
	\begin{aligned}\label{f8_estimate}
		|s_1(u^e,v_h)| &=\big| \xi_1\int_{Q^{S}}(\n_h\cdot \nabla u^e) (\n_h\cdot \nabla v_h)\dif \left(\x,t\right) \big| \\
		&=\big| \xi_1\int_{Q^{S}}\left((\n-\n_h)\cdot \nabla u^e\right) (\n_h\cdot \nabla v_h)\dif \left(\x,t\right) \big|\\
		&\lesssim \xi_1^{\frac12} h^{k_{\spa}}\norm{\nabla u^e}_{L^2(Q^{S})}\xi_1^{\frac12}\norm{\n_h\cdot\nabla v_h}_{L^2(Q^{S})}\lesssim h^{k_{\spa}} \norm{u}_{H^2(S)}\triplenorm v_h \triplenorm_{h,\Theta}.
	\end{aligned}
\end{align} 
For the other  stabilization term we use $ \int_{\Gamma(t)} u \dif s =0$, $\mu_h \dif s_h=\dif s$,  the estimate 
 \eqref{einminusmu_Space}, which then implies $|\overline{u^e}^n(t)| \lesssim
h^{k_{\spa}+1} \|u\|_{L^2(\Gamma(t))}$. Using this and the upper bound for \eqref{stabiparameterbedingung} for $\xi_2$:
\begin{align*}
 |s_2(u^e,v_h)| & = \xi_2 \big| \sum_{n=1}^N \int_{t_{n-1}}^{t_{n}}\overline{u^e}^n \overline{v}_h^n \dif t \big| \lesssim h^{k_{\spa}} \norm{u}_{L^2(S)}\norm{v_h}_{L^2(S_h)}.
\end{align*}
Combining the estimates for the different terms in \eqref{T1}-\eqref{T4} completes the proof.
\end{proof}

\subsection{Interpolation and discretization error bounds} \label{secinterpol}
In this section we discuss how bounds for the first and second term on the right-hand side of the Strang estimate \eqref{strangFormula} can be derived. These are based on interpolation error bounds and a continuity estimate.
We first address the interpolation error bounds. The interpolation operator ${I_\Theta^{k_\spa}:C^0(Q_\Theta^S)\rightarrow V_{h,\Theta}^{k_\spa}}$ is defined by $(I_\Theta^{k_\spa} v)\circ \Theta_h^n = I^{k_\spa}(v\circ \Theta_h^n)$ on $Q_n^S$, with $I^{k_\spa}$ the standard nodal interpolation operator in the space-time finite element space $V_h^{k_\spa}$ of polynomials of degree $k_\spa$ both in space and in time. For $I^{k_\spa}$ optimal error bounds are well-known. With fairly straightforward arguments, which are elaborated for the stationary case in \cite{grande2018analysis} and for the space-time case with $k_s=1$ in \cite{Diss_Sass}, one obtains optimal interpolation error bounds for $I_\Theta^{k_\spa}$, too. 
\begin{proposition} \label{PropInt}
For $u\in H^{k_\spa+1}(S)$ we have
\begin{align} \label{int11} 
& \|u^e - I_\Theta^{k_\spa}(u^e)\|_{L^2(Q_\Theta^S)} + h\|u^e - I_\Theta^{k_\spa}(u^e)\|_{H^1(Q_\Theta^S)}  \nonumber \\ & \qquad  \lesssim h^{k_\spa +1} \|u^e\|_{H^{k_\spa +1}(Q_\Theta^S) } \lesssim h^{k_\spa +\frac32}\|u\|_{H^{k_\spa +1}(S) },\\
  & \|u^e - I_\Theta^{k_\spa}(u^e)\|_{L^2(S_h)} + h\|u^e - I_\Theta^{k_\spa}(u^e)\|_{H^1(S_h)}   \lesssim h^{k_\spa +1} \|u\|_{H^{k_\spa +1}(S) },\label{int11A}
\end{align}
and 
	\begin{equation} \label{int12} \triplenorm u^e-I_\Theta^{k_\spa}(u^e)\triplenorm_{h,\Theta}  \lesssim h^{k_\spa} \|u\|_{H^{k_\spa +1}(S) } +h^{k_\spa +\frac12} \|u\|_{W^{1, \infty}(S) }  .
	 \end{equation}
\end{proposition}
\begin{proof}
 We do not include a proof of \eqref{int11}-\eqref{int11A}. The first result can be derived using the relation between $I_\Theta^{k_\spa}$ and $I^{k_\spa}$ and  standard error bounds for the latter. The result \eqref{int11A} can be derived using \eqref{int11} and an Hansbo type estimate as in \eqref{hansbo_var}.  We outline how \eqref{int12} can be derived using \eqref{int11}-\eqref{int11A}. Bounds for the first three terms in the norm $\triplenorm u^e-I_h(u^e)\triplenorm_{h,\Theta}$, cf. \eqref{energy_norm_ho_def}, immediately follow from \eqref{int11}-\eqref{int11A}. For the fourth term in this norm we apply the trace estimate $\|w\|_{L^2(\Gamma_h^N(t_N))} \lesssim \|w\|_{H^1(S_h)}$ and then use the result \eqref{int11A}. We finally consider the last term in the norm $\triplenorm u^e-I_\Theta^{k_{\spa}}(u^e)\triplenorm_{h,\Theta}$ that measures the size of the jump term $[u^e-I_\Theta^{k_\spa}(u^e)]_h^{n-1}$ on $\Gamma_h^n(t_{n-1})$. We introduce notation $\hat \Theta_h^n:= \Theta_h^n(\cdot, t_{n-1})$, $\hat \Theta_h^{n-1}:= \Theta_h^{n-1}(\cdot, t_{n-1})$. Note that $\hat \Theta_h^n: \, \Gamma_{\rm lin}(t_{n-1}) \to \Gamma_h^n(t_{n-1})$ is a bijection with inverse $\big(\hat \Theta_h^n\big)^{-1}=(\Theta_h^n)^{-1}(\cdot,t_{n-1})$. We denote the (nodal) interpolation error operators by $E_\Theta^{k_\spa}:= I - I_{\Theta}^{k_\spa}$, $E^{k_\spa}:=I- I^{k_\spa}$. Then we obtain, using the definition of $[\cdot]_h^{n-1}$, boundedness of the standard nodal interpolation operator in the maximum norm and \eqref{theta_psi_est}:
 \begin{equation} \label{PL} 
  \begin{split}
& \|[E_\Theta^{k_\spa}(u^e)]_h^{n-1}\|_{L^2(\Gamma_h^n(t_{n-1}))} \lesssim  \|[E_\Theta^{k_\spa}(u^e)]_h^{n-1}\circ \hat \Theta_h^n
  \|_{L^2(\Gamma_{\rm lin}(t_{n-1}))}  \\
  & \lesssim \|E_\Theta^{k_\spa}(u^e)\circ \hat \Theta_h^n- E_\Theta^{k_\spa}(u^e)\circ \hat \Theta_h^{n-1} 
  \|_{L^2(\Gamma_{\rm lin}(t_{n-1}))} \\
   & \lesssim \|E^{k_\spa}(u^e\circ \hat \Theta_h^n-u^e\circ \hat \Theta_h^{n-1})
  \|_{L^\infty({\Omega_n^\Gamma}_{|t=t_{n-1}})} \\ &  \lesssim \|u^e\circ \hat \Theta_h^n-u^e\circ \hat \Theta_h^{n-1}
  \|_{L^\infty({\Omega_n^\Gamma}_{|t=t_{n-1}})} \\
  &\lesssim \|\nabla u^e\|_{L^\infty({\Omega_n^\Gamma}_{|t=t_{n-1}})} h^{k_\spa+1} \lesssim h^{k_\spa+1}\|u\|_{W^{1,\infty}(S)}. 
 \end{split}
 \end{equation}
Using \cref{PL}, the last term in the norm $\triplenorm u^e-I_\Theta^{k_{\spa}}(u^e)\triplenorm_{h,\Theta}$ can be estimated by $h^{k_\spa+\frac12} \|u\|_{W^{1,\infty}(S)}$. Combining these results we obtain the result \eqref{int12}.
\end{proof}
\ \\[1ex]
As final step in  the error analysis we have to derive a bound for the term $B_h^{\stab}(u^e-I_\Theta^{k_s}(u^e),v_h)$ in the Strang estimate \eqref{strangFormula}. For this we can use arguments very similar to ones used above. We outline the key steps. First we apply partial integration, as in the proof of Theorem~\ref{thmconsisteny}, which results in the relation \eqref{pi_used_consis} with $u^e$ replaced by the interpolation error $E_\Theta^{k_\spa}(u^e):=u^e-I_\Theta^{k_s}(u^e)$. The first, second, third and fifth term on the right-hand side in \eqref{pi_used_consis} can estimated using the Cauchy Schwarz inequality and the interpolation error bounds in \eqref{int11A}, resulting in a bound $c h^{k_\spa} \|u\|_{H^{k_\spa +1}(S) } \triplenorm v_h \triplenorm_{h,\Theta}$. For the fourth term on the right-hand side in \eqref{pi_used_consis}, with the jump term $[E_\Theta^{k_\spa}(u^e)]_h^{n-1}$, we 
	use the estimates \eqref{PL} and \eqref{CustomInvIneq_bot}, which then results in a bound $c h^{k_\spa} \|u\|_{W^{1, \infty}(S) } \triplenorm v_h \triplenorm_{h,\Theta}$. The sixth term on the right-hand side in \eqref{pi_used_consis} contains the sum over the faces $F\in \mathcal{F}_I^n$.  For the interpolation error on these faces $\|E_\Theta^{k_\spa}(u^e)\|_{L^2(F)}$ we use \eqref{hansbo_var} and the interpolation  error bound \eqref{int11}. Combining this with \eqref{CustomInvIneq2} results in a bound $c h^{k_\spa} \|u\|_{H^{k_\spa +1}(S) } \triplenorm v_h \triplenorm_{h,\Theta}$ for the sixth term. It remains to derive bounds for the stabilization terms $s_1(E_\Theta^{k_\spa}(u^e), v_h)$ and 
	$s_2(E_\Theta^{k_\spa}(u^e),v_h)$. For the former one obtains with the Cauchy Schwarz in equality and \eqref{int11} the estimate $|s_1(E_\Theta^{k_\spa}(u^e), v_h)| \lesssim h^{k_\spa} \|u\|_{H^{k_\spa +1}(S) } \triplenorm v_h \triplenorm_{h,\Theta}$. For the latter stabilization term we use repeated Cauchy-Schwarz and \eqref{int11A}, which yields
	\[
	  |s_2(E_\Theta^{k_\spa}(u^e),v_h| \lesssim h^{-1} \|E_\Theta^{k_\spa}(u^e)\|_{L^2(S_h)} \|v_h\|_{L^2(S_h)}  \lesssim 
	  h^{k_\spa} \|u\|_{H^{k_\spa +1}(S) } \triplenorm v_h \triplenorm_{h,\Theta}.
	\]
Summarizing, we obtain for the remaining term in  the Strang estimate \eqref{strangFormula} the bound
\begin{equation} \label{Contest}
 \sup_{v_h\in V_{h,\Theta}^{k_\spa}}\frac{B_h^{\stab}(u_h-I_\Theta^{k_\spa}(u^e),v_h)}{\triplenorm v_h \triplenorm_{h,\Theta}} \lesssim  h^{k_\spa} \big(\|u\|_{H^{k_\spa +1}(S) }+\|u\|_{W^{1, \infty}(S) }\big).
\end{equation}
As a direct consequence of the Strang lemma \ref{LemmaStrang} and the estimates derived in Theorem~\ref{thmconsisteny}, Proposition~\ref{PropInt} and \eqref{Contest} we obtain the following discretizaton error bound.
\begin{theorem} \label{mainthm}
 For the  solution $u$ of \cref{prob_cont} we assume $u\in H^{k_\spa +1}(S) \cap W^{1,\infty}(S)$. Let $u_h$ be the solution of \Cref{discreteproblem}. The following discretization error estimate holds:
	\begin{equation} \label{estmain}
			\triplenorm u^e-u_h \triplenorm_{h,\Theta} \lesssim 
			h^{k_{\spa}} \big(\norm{u}_{H^{k_\spa+1} (S)}+\|u\|_{W^{1, \infty}(S)}+ h\norm{f}_{L^2(S)}\big).
	\end{equation}
\end{theorem}
\ \\
{\color{black} \begin{remark} \rm
 The energy norm $\triplenorm \cdot \triplenorm_{h,\Theta}$  contains a spatial gradient. Hence, for $\Delta t \sim h$ and using finite element polynomials of degree $k_s$ in space and time, the interpolation error in this norm is of order $\mathcal{O}(h^{k_s})$. The discretization error bound \eqref{estmain} is optimal in the sense that it has the same order of convergence $k_s$.   The bound confirms the  convergence results obtained in numerical experiments with $k_s=1,2,3,4$, cf.~\cite{sass2022}. One expects one order more if only the $L^2$ error (in space and time) is considered, cf. also the experiments in \cite{sass2022}.  Such a result, however, has not been rigorously derived, yet.
 
 The result in Theorem~\ref{mainthm} is the first  error bound for this type of higher order Eulerian space-time discretization method. The result, although sharp in the sense as explained above, is still very crude due to the assumptions $k_{g,s}=k_{g,q}=k_q=k_s$ and $\Delta t \sim h$. Several interesting accuracy issues are not addressed by our crude analysis. Here we mention two of these. If we keep the assumption $k_{g,s}=k_{g,q}=k_q=k_s$  but drop $\Delta t \sim h$ then, due to the fact that the energy norm does not contain a derivative with respect to $t$, for  the interpolation error term in \eqref{strangFormula} we expect,
 \[ \triplenorm u^e-I_\Theta^{k_\spa}(u^e)\triplenorm_{h,\Theta} \lesssim h^{k_s} + \Delta t^{k_s+1},
 \]
cf. the interpolation error bounds in \cite{HeimannLehrenfeld}. Such a higher order convergence with respect to $\Delta t$ in the energy norm is, however,  not observed in experiments. This might be due to  consistency error  terms that contain derivatives also in time direction and for which the optimal bound is of order $h^{k_s} + \Delta t^{k_s}$, cf. \eqref{estdiv}. A rigorous   theoretical explanation, however, is lacking. One other issue is  related to the well-known superconvergence effect of DG time discretization methods  at the nodes, cf. \cite[Theorem 12.3]{Thomee97}. We do not observe this effect in our numerical experiments. A satisfactory explanation for this is lacking.  
\end{remark}}
\ \\

\textcolor{black}{Concerning mass conservation properties of the discrete scheme there is no rigorous analysis.  Numerical experiments in \cite[Section 6]{Diss_Sass} and \cite{sass2022} show that the mass error $m(t):=|\int_{\Gamma_h^n(t)} u_h \, {\rm d}s_h- \int_{\Gamma_h^0(0)}u_h \, {\rm d}s_h|$, $t \in I_n$, is essentially independent of $n$ and for the case $k_{g,s}=k_s=k_{g,q}=k_q=1$ and $\Delta t \sim h$ we observe $m(t) \sim  h^2$.  
	}
	\ \\[1ex]
{\bf Acknowledgement.} The authors thank the German Research Foundation
(DFG) for financial support within the Research Unit ”Vector- and tensor valued
surface PDEs” (FOR 3013) with project no. RE 1461/11-2.
\section{Appendix: proof of Lemma~\ref{lemmeasure}} 
For $B\subset \mathbb{R}^3$ let $\f:B\rightarrow S_h^n$ be a local parametrization of $S_h^n$. The Lebesgue measure on $B$ is denoted by $\dif \zz$.
The Jacobian of $\f$ is denoted by $D_{\zz}\f=: \mathbf{F}: B \rightarrow \R^{4\times 3}$. The three columns of $\mathbf{F}$ are tangential to $S_h^n$. 
For given $\zz \in B$ we introduce a singular value decomposition $\mathbf{F}=\mathbf{U}\mathbf{\Sigma} \mathbf{V}^T$, with orthogonal matrices ${\mathbf{U}=(\mathbf{u}_1,\mathbf{u}_2,\mathbf{u}_3,\mathbf{u}_4)\in \R^{4\times 4}}$ and $\mathbf{V}\in \R^{3\times 3}$. The singular values are denoted by $\sigma_i$, $i=1,2,3$. We choose $\mathbf{u}_i$, $1 \leq i \leq 4$, such that $\mathbf{u}_4(\zz)=\n_{S_h}(\f(\zz))$.
We obtain	
\begin{equation}
	\dif \sigma_h=\sqrt{\det(\mathbf{F}^T \mathbf{F})}\dif \zz=\sqrt{\det(\mathbf{V}\mathbf{\Sigma}^T \mathbf{U}^T \mathbf{U}\mathbf{\Sigma} \mathbf{V}^T)}\dif \zz 
 = \prod_{i=1}^3\sigma_i \dif \zz.\label{sigmanh}
\end{equation}

We write $\delta_t\coloneqq \frac{\partial \delta}{\partial t}$ and $\n_t=\ddt{\n}{t}$ on $U$ and note that $\bn \cdot \bn_t=0$ holds.  
A local parametrization of $S^n$ is given by $\p\circ \f:B\rightarrow S^n$. Its Jacobian $\mathbf{A}:=D_{\zz} (\p\circ \f):B\rightarrow \R^{4\times 3}$ is given by 
\begin{equation*}
	\mathbf{A}=
	\begin{pmatrix}
		\gp-\delta\mathbf{H} & -\delta_t\n-\delta \n_t \\
		0& 1 
	\end{pmatrix} \mathbf{F}.
\end{equation*}
For ease of presentation we omit function arguments, noting that the functions $\bF$, $\bU$, $\bU_m$, $\bV$, $\bV_m$, $\bSigma$, $\bSigma_m$, $\sigma_i$, $\dif \zz$, some of which are introduced below, are evaluated on $B$, while the other functions are evaluated, via composition with $\mathbf{f}$, on $S_h^n$. We obtain the surface measure relation
$\dif \sigma =\sqrt{\det(\mathbf{A}^T\mathbf{A})}\dif \zz $ and thus
\begin{equation}
	\mu_h^S =\frac{\sqrt{\det(\mathbf{A}^T \mathbf{A})}}{ \prod_{i=1}^3\sigma_i }.\label{zwischenschritt_det_lemma}
\end{equation}
We will express $\det(\mathbf{A}^T \mathbf{A})$ in terms of the singular values $\sigma_i$ using basic linear algebra relations. Using the Schur complement formula it follows that for $\mathbf{G} \in \R^{n\times (n-1)}$, $\by \in \R^n, ~ \by \neq 0$,  we have
\begin{equation} \label{elim1}
  \det \begin{pmatrix}\mathbf{G} & \by \end{pmatrix}^2 = \det \begin{pmatrix}\mathbf{G}^T\mathbf{G} & \mathbf{G}^T\by \\
  \by^T\mathbf{G} & \by^T \by \end{pmatrix} =\by \cdot \by\,  \det ( \mathbf{G}^T\mathbf{G} -(\by \cdot \by)^{-1} \mathbf{G}^T \by \by^T\mathbf{G}). 
\end{equation}
Using this and $\mathbf{A}^T \begin{pmatrix}\n\\\delta_t \end{pmatrix}=0$ we get
\begin{equation} \label{elim2} 
 \det (\mathbf{A}^T \mathbf{A}) = (1+\delta_t^2)^{-1}\det\begin{pmatrix}
			\mathbf{A} & \begin{matrix}
				\n\\\delta_t
			\end{matrix}
		\end{pmatrix}^2.
\end{equation}
An elementary calculation yields
\begin{equation*}
	\det\begin{pmatrix} \mathbf{I}-\delta \mathbf{H} & -\delta \n_t  \\ \n^T \delta_t & 1+\delta_t^2 \end{pmatrix}=(1+\delta_t^2)\det(\mathbf{I}-\delta \mathbf{H})=(1+\delta_t^2)\prod_{i=1}^{2}(1-\delta \kappa_i).
\end{equation*}
Using this we obtain
\begin{align}
	\det \begin{pmatrix}
			\mathbf{A} & \begin{matrix}
				\n\\\delta_t
			\end{matrix}
		\end{pmatrix}&
	=\det\begin{pmatrix} \mathbf{I}-\delta \mathbf{H} & -\delta \n_t \nonumber \\ \delta_t\n^T  & 1+\delta_t^2 \end{pmatrix}\det\begin{pmatrix}
		\begin{pmatrix}
			\mathbf{P} & -\delta_t\n\\0&1\end{pmatrix}\mathbf{F} &\n_0
	\end{pmatrix} \\ & = (1+\delta_t^2)\prod_{i=1}^{2}(1-\delta \kappa_i)\det\begin{pmatrix}
		\begin{pmatrix}
			\mathbf{P} & -\delta_t\n\\0&1\end{pmatrix}\mathbf{F} &\n_0
	\end{pmatrix} .\label{zweidets}
\end{align}
We apply \eqref{elim1} with
	$\mathbf{G}= \begin{pmatrix}
		\mathbf{P} & -\delta_t\n\\0&1\end{pmatrix}\mathbf{F}$, which yields
\begin{align*}
	&\det\begin{pmatrix}
		\begin{pmatrix}
			\mathbf{P} & -\n\delta_t\\0&1\end{pmatrix}\mathbf{F} &\n_0
	\end{pmatrix}^2=\det \left(\mathbf{G}^T \mathbf{G}-\mathbf{G}^T \n_0\n_0^T\mathbf{G}\right)=\det\left(\mathbf{G}^T \begin{pmatrix}\mathbf{P} &0\\0&1\end{pmatrix}\mathbf{G}\right)\\
	&\qquad =\det\left(\mathbf{F}^T\begin{pmatrix}\mathbf{P} &0\\0&1\end{pmatrix}\mathbf{F}\right)=\det\left( \mathbf{\Sigma}^T \mathbf{U}^T \begin{pmatrix}\mathbf{P} &0\\0&1\end{pmatrix} \mathbf{U}\mathbf{\Sigma}\right).
\end{align*}
We define $\mathbf{U}_m\coloneqq (\mathbf{u}_1,\mathbf{u}_2,\mathbf{u}_3) \in \R^{4\times 3}$ and $\mathbf{\Sigma}_m:={ \rm diag}(\sigma_1,\sigma_2,\sigma_3) \in \R^{3\times 3}$, hence, $\mathbf{U}\mathbf{\Sigma}=\mathbf{U}_m\mathbf{\Sigma}_m$. Thus we get
\begin{equation*}
	\det\begin{pmatrix}
		\begin{pmatrix}
			\mathbf{P} & -\n\delta_t\\0&1\end{pmatrix}\mathbf{F} &\n_0
	\end{pmatrix} =\det\left(\mathbf{U}_m^T \begin{pmatrix}\mathbf{P} &0\\0&1\end{pmatrix}\mathbf{U}_m \right)^\frac12\prod_{i=1}^{3}\sigma_i.
\end{equation*}
Combining this with the results in \eqref{zwischenschritt_det_lemma}, \eqref{elim2} and \eqref{zweidets} we get
\begin{equation} \label{55}
  \mu_h^S= (1+\delta_t^2)^\frac12 \prod_{i=1}^{2}(1-\delta \kappa_i)\det\left(\mathbf{U}_m^T \begin{pmatrix}\mathbf{P} &0\\0&1\end{pmatrix}\mathbf{U}_m \right)^\frac12.
\end{equation}
Now note
\begin{align} \label{66}
	\begin{aligned}
		\det\left(\mathbf{U}_m^T \begin{pmatrix}\mathbf{P} &0\\0&1\end{pmatrix}\mathbf{U}_m \right) =\det\left( \mathbf{I}- \mathbf{U}_m^T\n_0 \n_0^T\mathbf{U}_m^T\right)= 1- (\mathbf{U}_m^T \n_0)\cdot(\mathbf{U}_m^T \n_0).
	\end{aligned}
\end{align}
From $\bn_0= \sum_{i=1}^3 (\bn_0\cdot\bu_i) \bu_i + (\bn_0\cdot \bn_{S_h}) \bn_{S_h}$ it follows that
\[
  1= \sum_{i=1}^3 (\bn_0\cdot \bu_i)^2 + (\bn_0\cdot \bn_{S_h})^2=(\mathbf{U}_m^T \n_0)\cdot(\mathbf{U}_m^T \n_0)+ (\bn_0\cdot \bn_{S_h})^2.
\]
Using this and \eqref{55}-\eqref{66}
we get
\[ \mu_h^S= (1+\delta_t^2)^\frac12 \prod_{i=1}^{2}(1-\delta \kappa_i) \, \bn_0\cdot \bn_{S_h} 
\]
Using the  identity $\delta_t=-\w^e\cdot \n = V^e$ on $S_h^n$ completes the proof.

\bibliographystyle{siam}
\bibliography{literatur}

\end{document}